\documentclass[11pt]{amsart}
\usepackage[11pt]{extsizes}

\usepackage[all]{xy}
\usepackage[english]{babel}
\usepackage[utf8]{inputenc}
\usepackage[T1]{fontenc}
\usepackage{amssymb}
\usepackage{amsmath}
\usepackage{amsfonts}
\usepackage{amsthm}
\usepackage{stmaryrd}
\usepackage{mathabx}
\usepackage[pdftex]{color}
\usepackage[top=2.75cm, bottom=2.75cm, left=2.25cm, right=2.25cm]{geometry}

\usepackage[bookmarks=true,bookmarksopen=true,linkbordercolor={1 1 1}]{hyperref}
\hypersetup{
    colorlinks,
    citecolor=blue,
    filecolor=blue,
    linkcolor=blue,
    urlcolor=blue
}

\theoremstyle{plain}
\newtheorem{theorem}{Theorem}
\newtheorem{proposition}{Proposition}
\newtheorem{lemma}{Lemma}
\newtheorem{corollary}{Corollary}

\theoremstyle{definition}
\newtheorem{definition}{Definition}

\theoremstyle{remark}

\newtheorem{remark}{Remark}

\renewcommand\ker{\operatorname{Ker}}
\newcommand\ns{\operatorname{NS}}
\newcommand\taut{\operatorname{Taut}}
\newcommand\End{\operatorname{End}}
\newcommand\pic{\operatorname{Pic}}
\newcommand\gal{\operatorname{Gal}}
\newcommand\im{\operatorname{Im}}

\newcommand\aut{\operatorname{Aut}}

\newcommand\id{\operatorname{Id}}
\renewcommand\hom{\operatorname{Hom}}
\newcommand\dA{{\rm{A}}}
\newcommand\dK{{\rm{K}}}

\begin{document}

\title{Tautological rings on Jacobian varieties of curves with automorphisms}

\author{Thomas Richez}
\address{Université de Strasbourg, CNRS, IRMA UMR 7501, F-67000 Strasbourg, France           
}

\email{richez@math.unistra.fr}

\maketitle

\footnotetext{The final publication is available at Springer via http://dx.doi.org/10.1007/s10711-017-0262-9}

\begin{abstract}
Let $J$ be the Jacobian of a smooth projective complex curve $C$ which admits non-trivial automorphisms, and let $\dA(J)$ be the ring of algebraic cycles on $J$ with rational coefficients modulo algebraic equivalence. We present new tautological rings in $\dA(J)$ which extend in a natural way the tautological ring studied by Beauville in \cite{MR2041776}. We then show there exist tautological rings induced on special complementary abelian subvarieties of $J$.
\end{abstract}

\vspace*{5pt}


\vspace*{5pt}


\vspace*{10pt}

\section{Introduction} \label{sec:intro}

In this paper we consider varieties over $\mathbb{C}$. Let $X$ be an abelian variety of dimension $g \geq 1$. We denote by $m$ its group law and by $\widehat{X}$ the dual variety. We consider the ring $\dA^\cdot(X)$ of algebraic cycles on $X$ with rational coefficients modulo algebraic equivalence. Beauville showed in \cite{MR826463} that there exists a bigraduation on $\dA(X)$. Specifically, we have
\begin{displaymath}
\dA^p(X) = \bigoplus_{s=p-g}^p \dA^p(X)_{(s)}
\end{displaymath}
where $p$ refers to the codimension grading and $s$ refers to eigenspaces of the operators $k_*$ and $k^*$ induced by the homotheties $k = k_X$ on $X$ for any $k \in \mathbb{Z}$. These eigenspaces are characterized by $x \in \dA^p(X)_{(s)}$ if and only if for all $k \in \mathbb{Z}$, $k^* x= k^{2p - s} x$ (or equivalently $k_* x = k^{2g- 2p + s} x$). Note that this bigraduation is compatible with the intersection and Pontryagin products on $X$ denoted respectively by $\cdot : \dA^p(X)_{(s)} \times \dA^q(X)_{(t)} \to \dA^{p+q}(X)_{(s+t)}$ and $* : \dA^p(X)_{(s)} \times \dA^q(X)_{(t)} \to \dA^{p+q-g}(X)_{(s+t)}$. An important tool to study this structure on $\dA(X)$ which will play a major role in the sequel is the Fourier transform $\mathcal{F}_X: \dA(X) \to \dA(\widehat{X})$ on $X$. This map is defined as follows. Consider the Poincaré line bundle $\mathcal{P}_{X \times \widehat{X}}$ on $X \times \widehat{X}$ and its cycle class  $l_{X \times \widehat{X}} := c_1(\mathcal{P}_{X \times \widehat{X}})$ in $\dA^1(X \times \widehat{X})$. For any cycle $x \in \dA(X)$, we put $\mathcal{F}_X(x) := p_{2*} (p_1^* x \cdot e^{l_{X \times \widehat{X}}})$ where $p_1$ and $p_2$ are the natural projections of $X \times \widehat{X}$ to $X$ and $\widehat{X}$ respectively. Recall the following important facts (see \cite{MR726428}):
\begin{enumerate}
\item Identifying $X$ with its bidual variety $\widehat{\widehat{X}}$ (as we will always do), we get a Fourier transform $\mathcal{F}_{\widehat{X}}: \dA(\widehat{X}) \to \dA(X)$ on $\widehat{X}$. It satisfies the relation
\begin{displaymath}
\mathcal{F}_{\widehat{X}} \circ \mathcal{F}_X = (-1)^g (-1_X)^*.
\end{displaymath}
\item For all cycles $x,y$ on $X$, we have
\begin{displaymath}
\mathcal{F}_X(x \cdot y) = (-1)^g \mathcal{F}_X(x) * \mathcal{F}_X(y) \qquad \text{and} \qquad \mathcal{F}_X(x * y) = \mathcal{F}_X(x) \cdot \mathcal{F}_X(y).
\end{displaymath}
\end{enumerate}
The reader should refer to \cite{MR726428} for an overview of many other properties of the Fourier transform. In \S \ref{sec:preliminaries} we present slight generalizations of these properties. This section will be used in \S \ref{sec:ringY} and \S \ref{sec:ringZ} when we will work with non principal polarizations.

In \S \ref{sec:functoriality} we consider a smooth projective curve $C$ of genus $g(C) = g \geq 1$ with Jacobian $J = J(C)$. We fix a rational point $P$ on $C$ to embed the curve in its Jacobian via the usual map $f^P: C \hookrightarrow J$ defined on points by $Q \mapsto \mathcal{L}_C(Q - P)$. This map allows us to consider the cycle class defined by $C$, and still denoted by $C$, in $\dA^{g-1}(J)$. Note that this class does not depend on the choice of $P$ since we are working modulo algebraic equivalence.
Let $\mathcal{J} \subset \dA(X)$ be a family of cycles on $X$. We denote by $\taut_X(\mathcal{J})$ the tautological ring generated by $\mathcal{J}$, that is to say the smallest $\mathbb{Q}$-vector subspace of $\dA(X)$ containing $\mathcal{J}$ and closed under natural operations on $\dA(X)$; namely intersection and Pontryagin products, and operators $k_*, k^*$ for all $k \in \mathbb{Z}$.

We assume that the reader is familiar with paper \cite{MR2041776} in which Beauville studied the tautological ring $R(C;J) := \taut_J(\{C\})$. He proved that the $\mathbb{Q}$-algebra $R(C;J)$ is generated for the intersection product by the classes
\begin{displaymath}
w^i = \frac{1}{(g-i)!} C^{*(g-i)} \in \dA^i(J), \qquad i \in \llbracket 0,g \rrbracket
\end{displaymath}
of the subvarieties $W_{g-i}$ parametrizing effective divisors on $C$ of degree $g-i$. Another system of generators is given by Newton polynomials in the $w^i$, denoted by $N^i(w) \in \dA^i(J)_{(i-1)}$. When $R(C ; J)$ is endowed with its structure of algebra for the Pontryagin product, a set of generators is given by the Fourier transforms of the $N^i(w)$, which are (up to a sign) the components $C_{(i)} \in \dA^{g-1}(J)_{(i)}$ appearing in Beauville's decomposition of $C \in \dA^{g-1}(J)$. The aim of \S \ref{sec:functoriality} is to clarify the functorial behaviour of this tautological ring $R(C ; J)$. In Section \ref{sec:functoriality} we consider a finite morphism of curves $f: C \to C'$ and we explain the action of the induced morphism $f^*: J(C') \to J(C)$ and the Albanese morphism $N_f: J(C) \to J(C')$ on $R(C ; J(C))$ and $R(C' ; J(C'))$. For a morphism of curves $f$, the abelian subvariety $Y := \im(f^*)$ of $J(C)$ with canonical embedding $\iota_Y: Y \hookrightarrow J(C)$ plays a crucial role. Indeed $Y$ is isogenous to $J(C')$ via the corestriction map $j = f^*: J(C') \to Y$. We will also associate to $Y$ (as we will do for any abelian subvariety of $J(C)$) its norm-endomorphism $N_Y : J(C) \to J(C)$ and the map $\psi_Y \in \hom(J(C),Y)$ defined by $N_Y = \iota_Y \circ \psi_Y$ (see \cite[\S 5.3]{MR2062673}). When $f : C \to C' \simeq C/\langle\sigma\rangle$ is a cyclic Galois covering for some $\sigma \in \aut(C)$ of finite order, one highlights naturally in $\dA(J(C))$ some cycle classes of the form $P(\sigma)_* C$ where $P(\sigma) \in \mathbb{Z}[\sigma]$ is a polynomial in the automorphism $\sigma$ or more accurately a polynomial in the Albanese morphism still denoted by $\sigma = N_{\sigma} \in \aut(J(C))$.

This leads us to \S \ref{sec:ringJacobian} where we consider a curve $C$ with a finite automorphism group $G \subset \aut(C)$. We will prove the following main result:

\begin{theorem} \label{theoRsigmiG}
Let $C$ be a smooth projective complex curve of genus $g \geq 1$ and $G$ a finite group of automorphisms of $C$. Then the tautological ring 
\begin{displaymath}
R_G(C ; J) := \taut_J \Big(\{\pi_* C \in \dA(J) ~|~ \pi \in  \mathbb{Z}[G] \subset \End(J) \} \Big)
\end{displaymath}
is generated as $\mathbb{Q}$-subalgebra of $\dA(J)$
\begin{enumerate}
\item for the intersection product by all $\pi^* N^i(w)$,
\item for the Pontryagin product by all $\pi_* C_{(i-1)}$
\end{enumerate}
with $\pi \in \mathbb{Z}[G]$ and $i \in \llbracket 1, g-1 \rrbracket$.
\end{theorem}

In case of a cyclic automorphism group $G = \langle\sigma\rangle$, we will put $R_\sigma(C ; J) := R_{\langle\sigma\rangle}(C ; J)$. Furthermore, each subgroup $K$ of $G$ determines a subtautological ring $R_K(C ; J) \subset R_G(C ; J)$. For example, with $K = \{\id\}$ we get $R(C ; J) \subset R_G(C ; J)$. Actually, the tautological ring $R_G(C ; J)$ is the smallest $\mathbb{Q}$-algebra extension of $R(C ; J)$ which is stable under intersection product, Pontryagin product and pull-backs and push-forwards by elements in $\mathbb{Z}[G] \subset \End(J)$. This is a very natural characterization which may have been chosen at first to define these tautological rings:

\begin{corollary} \label{corPlusPetiteExtension}
The tautological ring $R_G(C ; J)$ is the smallest $\mathbb{Q}$-algebra extension of $R(C ; J)$ for the intersection product (resp. Pontryagin product) which is stable under pull-backs (resp. push-forwards) by polynomials in $\mathbb{Z}[G] \subset \End(J)$.
\end{corollary}

Now let us stress why the adjective \emph{tautological} is still appropriate to such rings $R_G(C ; J)$. If one considers a curve without non-trivial automorphism, we are interested in the smallest $\mathbb{Q}$-vector subspace of $\dA(J)$ which contains the cycle class $C$, and closed under both products, pull-backs and push-forwards by scalars in $\mathbb{Z} \subset \End(J)$ (that is constant polynomials). This ring is precisely Beauville's tautological ring $R(C ; J)$. But if $C$ has a non-trivial automorphism group $G$, the same natural idea leads us to study the smallest $\mathbb{Q}$-vector subspace of $\dA(J)$ which contains the class $C$, and closed under both products, pull-backs and push-forwards by elements in $\mathbb{Z}[G]$; that is $R_G(C ; J)$. 
Besides, having all these tautological and subtautological rings associated to groups and subgroups of automorphisms strengthens the following idea: for a Jacobian variety with non-trivial automorphisms, the ring $\dA(J)$ carries a much richer structure than that of a generic Jacobian; which is already an interesting fact in itself.

In the next section, that is \S \ref{sec:ringY}, we will explore the link between tautological rings of $J(C/\langle\sigma\rangle)$ and $J(C)$. These rings are closely related as pointed out in

\begin{theorem} \label{theoRsigmaisYGeneralise}
Let $f: C \to C' \simeq C/\langle\sigma\rangle$ be a $n$-cyclic Galois covering associated to an automorphism $\sigma \in \aut(C)$ of order $n \in \mathbb{N}^*$. We consider a finite group of automorphisms $G \subset \aut(C)$ and we suppose that each $g \in G$ commutes with $\sigma$ so that there is an automorphism $\widetilde{g} \in \aut(C')$ satisfying the relation $f \circ g = \widetilde{g} \circ f$. We denote by $\widetilde{G}$ the image of $G$ in $\aut(C')$. Then the tautological ring
\begin{align*}
R_G(\psi_{Y*} C ; Y)
:= \taut_Y \Big(\{ \pi_* \psi_{Y*} C \in \dA(Y) ~|~ \pi \in \mathbb{Z}[G]\} \Big)
\end{align*}
 is generated as $\mathbb{Q}$-subalgebra of $\dA(Y)$
\begin{enumerate}
\item for the intersection product by all $\pi^*\iota_Y^* N^{i+1}(w) = \iota_Y^* \pi^* N^{i+1}(w)$,
\item for the Pontryagin product by all $\pi_* \psi_{Y*} C_{(i)} = \psi_{Y*} \pi_* C_{(i)}$
\end{enumerate}
with $\pi \in \mathbb{Z}[G]$ and $i \in \llbracket 0, g(C')-1 \rrbracket$.
Therefore, the isogeny $j$ induces an isomorphism (as $\mathbb{Q}$-vector spaces)
\begin{align*}
R_{\widetilde{G}}(C' ; J(C')) 
& \simeq j_* R_{\widetilde{G}}(C' ; J(C')) = \taut_Y \Big(j_* R_{\widetilde{G}}(C' ; J(C')) \Big) \\
& = \iota_Y^* R_G(C ; J(C)) = \psi_{Y*} R_G(C ; J(C))  = R_G(\psi_{Y*} C ; Y). 
\end{align*}
\end{theorem}

The last part of this article, that is \S \ref{sec:ringZ}, is dedicated to tautological rings induced on the natural abelian subvariety $Z$ of $J(C)$, the complementary abelian variety to $Y$ with respect to the Theta polarization on $J(C)$ (see \cite[Section 5.3]{MR2062673}). We will prove

\begin{theorem} \label{theoRsigmaisZGeneralise}
Let $f: C \to C' \simeq C/\langle\sigma\rangle$ be a $n$-cyclic Galois covering associated to an automorphism $\sigma\in \aut(C)$ of order $n \in \mathbb{N}^*$. We consider a finite group of automorphisms $H \subset \aut(C)$ and we suppose that $\sigma \in H$ is central in $H$. Then the tautological ring
\begin{align*}
R^{\sigma}_H(\psi_{Z*} C ; Z)
:= \taut_Z \Big(\{ \pi_* \psi_{Z*} C \in \dA(Z) ~|~ \pi \in \mathbb{Z}[H]\} \Big)
\end{align*}
is generated as $\mathbb{Q}$-subalgebra of $\dA(Z)$
\begin{enumerate}
\item for the intersection product by all $\pi^* \iota_Z^* N^i(w) = \iota_Z^* \pi^* N^i(w)$,
\item for the Pontryagin product by all $\pi_* \psi_{Z*} C_{(i-1)} = \psi_{Z*} \pi_* C_{(i-1)}$
\end{enumerate}
with $\pi \in \mathbb{Z}[H]$ and all $i \in \llbracket 1,\dim Z-1 \rrbracket$. In other words,
\begin{displaymath}
R^{\sigma}_H(\psi_{Z*} C ; Z) = \iota_Z^* R_H(C ; J(C)) = \psi_{Z*} R_H(C ; J(C)).
\end{displaymath}
\end{theorem}

In particular, considering the case of a cyclic automorphism group $H = \langle\sigma\rangle$ leads to:

\begin{theorem} \label{theoRsigmaisZGeneraliseSigma0}
Let $f: C \to C' \simeq C/\langle\sigma\rangle$ be a $n$-cyclic Galois covering associated to an automorphism $\sigma\in \aut(C)$ of order $n \in \mathbb{N}^*$. Then the tautological ring
\begin{align*}
& R_{\sigma}(\psi_{Z*} C ; Z)
:= \taut_Z \Big(\{P(\sigma)_* \psi_{Z*} C \in \dA(Z) ~|~ P \in \mathbb{Z}[X]\} \Big)
\end{align*}
is generated as $\mathbb{Q}$-subalgebra of $\dA(Z)$
\begin{enumerate}
\item for the intersection product by all $P(\sigma)^* \iota_Z^* N^i(w) = \iota_Z^* P(\sigma)^* N^i(w)$,
\item for the Pontryagin product by all $P(\sigma)_* \psi_{Z*} C_{(i-1)} = \psi_{Z*} P(\sigma)_* C_{(i-1)}$
\end{enumerate}
with $P \in \mathbb{Z}[X]$ and all $i \in \llbracket 1,\dim Z-1 \rrbracket$. In particular,
\begin{displaymath}
R_{\sigma}(\psi_{Z*} C ; Z) = \iota_Z^* R_{\sigma}(C ; J(C)) = \psi_{Z*} R_{\sigma}(C ; J(C)).
\end{displaymath}
\end{theorem}

This theorem \ref{theoRsigmaisZGeneraliseSigma0} yields a generalization of a theorem proved by Arap \cite{MR2923946} who gave the analogue in $\dA(Z)$ of Beauville's tautological ring $R(C ; J(C))$ in the special case where $Z$ is a Prym variety. That is essentially when $f: C \to C'$ is of degree $2$ and either étale or ramified in exactly two points (see \cite[Theorem 12.3.3]{MR2062673}). We finish with a few examples which provide a full explicit structure for the algebra $R_{\sigma}(\psi_{Z*} C ; Z) \subset \dA(Z)$ when $\sigma$ is of order $2$ and $C$ is a $k$-gonal curve with $k \in \{2,3,4,5\}$.

\section{Preliminaries} \label{sec:preliminaries}

The Fourier transform on abelian varieties will be central in almost all following results. Therefore we start with some properties of this map. The following proposition is a slight but useful generalization of Beauville's result (\cite[Proposition 3.(iii)]{MR726428} or \cite[Proposition 16.3.4]{MR2062673}). It will help us to work with Fourier transform and pull-backs or push-forwards by arbitrary morphisms of abelian varieties.

\begin{remark}
By definition a morphism of abelian varieties respects the group structure.
\end{remark}

\begin{proposition} \label{propCommutationFourierPullBackPushForward}
Let $X,Y$ be two abelian varieties and $\alpha: Y \to X$ a morphism of abelian varieties. Then
\begin{displaymath}
\mathcal{F}_X \circ \alpha_* = \widehat{\alpha}^* \circ \mathcal{F}_Y \qquad \text{and} \qquad
\mathcal{F}_Y \circ \alpha^* = (-1)^{\dim X - \dim Y} \widehat{\alpha}_* \circ \mathcal{F}_X.
\end{displaymath}
In particular, if $\alpha$ is an isogeny or if $X = Y$, we also have $\mathcal{F}_Y \circ \alpha^* = \widehat{\alpha}_* \circ \mathcal{F}_X$.
\end{proposition}

\begin{proof}
The proof of the first equality is analogous to \cite[Proposition 16.3.4 (a)]{MR2062673} which works without the isogeny hypothesis in the statement. The proof of the second equality is similar to \cite[Proposition 16.3.4 (b)]{MR2062673} with the particularity that the exponent of $(-1)$ is due to the difference of the dimensions of the abelian varieties involved. It appears with the inversion formulas for the Fourier transforms on $X$ and $Y$.
\end{proof}

The different results presented in this paper involve polarized (but not necessarily principally polarized) abelian varieties $(X,\xi)$. For such a polarization, we consider the isogeny
\begin{displaymath}
\varphi_\xi: X \to \pic^0(X) \simeq \widehat{X}
\end{displaymath}
given on points by
\begin{displaymath}
\varphi_\xi(x) := t_x^* \mathcal{L}_X(\xi) \otimes \mathcal{L}_X(\xi)^\vee
\end{displaymath}
where $\mathcal{L}_X(\xi)^\vee$ denotes (the class of) the dual invertible sheaf associated to (the class of) the ample divisor $\xi$. It is known that there exists an inverse isogeny up to scalar, denoted by $\psi_\xi \in \hom(\widehat{X},X)$. These morphisms satisfy relations $\psi_\xi \circ \varphi_\xi = n_X$ and $\varphi_\xi \circ \psi_\xi = n_{\widehat{X}}$ for some $n \in \mathbb{N}^*$. Recall that the dual map of $\varphi_\xi$ satisfies $\widehat{\varphi_\xi} = \varphi_\xi$ (\cite[Corollary 2.4.6]{MR2062673})  and thus $\widehat{\psi_\xi} = \psi_\xi$ too. Having said that, we will often consider the maps $\varphi_{\xi}^* \mathcal{F}_X: \dA(X) \to \dA(X)$ or $\psi_{\xi*} \mathcal{F}_X: \dA(X) \to \dA(X)$ instead of $\mathcal{F}_X$.

The following proposition give us some properties of the operator $\varphi_{\xi}^* \mathcal{F}_X$ as in \cite[\S 2.4 -- 2.7]{MR2041776}. It allows us to link (more deeply) the Fourier transform on $X$ and the Pontryagin product.

\begin{proposition} \label{propFourierXFonctionXi}
Keeping above notations, we consider the operator $\mathcal{F} := \varphi_{\xi}^* \mathcal{F}_X$. It satisfies the following properties:
\begin{enumerate}
\item Inversion formula: $\mathcal{F} \circ \mathcal{F} = \deg(\varphi_\xi) (-1)^{\dim X} (-1_X)^*$.
\item We have for all $x,y \in \dA(X)$, $\mathcal{F}(x * y) = \mathcal{F}(x) \cdot \mathcal{F}(y)$ and $\mathcal{F}(x \cdot y) = \frac{(-1)^{\dim X}}{\deg(\varphi_\xi)} \mathcal{F}(x) * \mathcal{F}(y)$.
\item $\mathcal{F}(\dA^p(X)_{(s)}) = \dA^{\dim X - p + s}(X)_{(s)}$.
\item Let $x \in \dA(X)$. We put $\overline{x} := (-1)^*x$. Then $\mathcal{F} (x) = e^{-\xi}\cdot((\overline{x} \cdot e^{-\xi}) * e^{\xi}) \in \dA(X)$.
\end{enumerate}
\end{proposition}

\begin{proof}
Because this proposition is similar to \cite[\S 2.4 -- 2.7]{MR2041776} we only prove the first statement in order to emphasize how the operator $\varphi_{\xi}^*$ effects the results. It is known that $\mathcal{F}_{\widehat{X}} \circ \mathcal{F}_X = (-1)^{\dim X} (-1_X)^*$. Therefore, using that $\varphi_{\xi*} \varphi_\xi^* = \deg(\varphi_\xi)$, $\widehat{\varphi_\xi} = \varphi_\xi$ and the compatibility between $\mathcal{F}_X$ and isogenies (Proposition \ref{propCommutationFourierPullBackPushForward}), we get $\mathcal{F} \circ \mathcal{F} = \varphi_{\xi}^* \mathcal{F}_X \varphi_{\xi}^* \mathcal{F}_X
 = \mathcal{F}_{\widehat{X}} \varphi_{\xi*} \varphi_{\xi}^* \mathcal{F}_X = \deg(\varphi_\xi) (-1)^{\dim X} (-1_X)^*$.
\end{proof}

\begin{remark}
In his paper \cite[\S 2.3]{MR2041776} Beauville uses the relation $l_{X \times X} = p^* \xi + q^* \xi - m^* \xi$ for the class of the Poincaré line bundle on $X \times X$. This equality is given by a different choice in the definition of the isogeny $\varphi_\xi$ associated to $\xi$. Namely, he uses the polarization $- \varphi_{\xi}$. This explains the difference between the last statement of Proposition \ref{propFourierXFonctionXi} and \cite[\S 2.7]{MR2041776}.
\end{remark}

We use this proposition to deduce the following corollary (inspired by \cite[Lemme 1]{MR726428} or \cite[Proposition 5]{MR726428}). It will be used only once, to prove Proposition \ref{propEquivStabFourierPontryaginGeneralisePhi}.

\begin{corollary} \label{corCalculF(exp(xi))}
Let $\xi \in \dA^1(X)_{(0)}$ be a polarization on an abelian variety $X$. Then
\begin{displaymath}
\varphi_{\xi}^* \mathcal{F}_X(e^\xi) = \chi(\xi) e^{-\xi}
\end{displaymath}
where $\chi$ denotes the Euler characteristic.
Accordingly,
\begin{displaymath}
e^{\xi} 
= \frac{(-1)^{\dim X}}{\chi(\xi)} \varphi_{\xi}^* \mathcal{F}_{X}(e^{-\xi}).
\end{displaymath}
\end{corollary}

\begin{proof}
Thanks to Proposition \ref{propFourierXFonctionXi} (4) and since $(-1)^* \xi = \xi$, we have
\begin{align*}
\varphi_\xi^* \mathcal{F}_X(e^\xi) 
& = e^{-\xi} \cdot ((e^{(-1)^*\xi} \cdot e^{-\xi}) * e^\xi)
= e^{-\xi} \cdot ((e^{\xi} \cdot e^{-\xi}) * e^\xi)
= e^{-\xi} \cdot ([X] * e^\xi).
\end{align*}
Thus, for codimension reasons and by using the Riemann-Roch theorem for abelian varieties (see \cite[p150]{MR2514037}), we obtain
\begin{align*}
\varphi_\xi^* \mathcal{F}_X(e^\xi)
& = e^{-\xi} \cdot \left([X] * \frac{1}{(\dim X)!} \xi^{\dim X} \right)
= \chi(\xi) e^{-\xi} \cdot \left([X] * [o] \right) 
= \chi(\xi) e^{-\xi} \cdot [X]
= \chi(\xi) e^{-\xi}
\end{align*}
where $[o]$ denotes the class of a point in $X$. Hence the first part of the corollary. Moreover, using the inversion formula (see Proposition \ref{propFourierXFonctionXi} (1)), we get
\begin{displaymath}
\deg(\varphi_\xi) (-1)^{\dim X} (-1)^* e^{\xi} 
= \varphi_\xi^* \mathcal{F}_X \varphi_\xi^* \mathcal{F}_X(e^\xi)
= \chi(\xi) \varphi_\xi^* \mathcal{F}_X (e^{-\xi}).
\end{displaymath}
Since $\deg(\varphi_\xi) = \chi(\xi)^2$ (see \cite[p150]{MR2514037}) and $(-1)^*e^{\xi} = e^{\xi}$ (because $\xi$ is symmetric), we obtain the second part of this corollary.
\end{proof}

\begin{proposition} \label{propEquivStabFourierPontryaginGeneralisePhi}
Let $T$ be a bigraded $\mathbb{Q}$-subalgebra (for the intersection product) of $\dA(X)$. We suppose that $T$ contains the class of the polarization $\xi$ on $X$. The following statements are equivalent:
\begin{enumerate}
\item $T * T \subset T$.
\item $\varphi_{\xi}^* \mathcal{F}_X(T) \subset T$.
\item $\varphi_{\xi}^*\mathcal{F}_{X} \varphi_{\xi}^* \mathcal{F}_X(T) \subset \varphi_{\xi}^* \mathcal{F}_X(T)$.
\item $\xi \cdot \varphi_{\xi}^* \mathcal{F}_X(T) \subset \varphi_\xi^* \mathcal{F}_X(T)$.
\end{enumerate}
\end{proposition}

\begin{proof}
Let us note $\mathcal{F} := \varphi_{\xi}^* \mathcal{F}_X  : \dA(X) \to \dA(X)$.

\paragraph{$(1) \Leftrightarrow (2)$}
If we assume that $T$ is stable under Pontryagin product, it is then clear according to Proposition \ref{propFourierXFonctionXi}, statement $4$, that for all $x \in T$, $\mathcal{F}(x) \in T$. Indeed, since $T$ is bigraded and $\xi \in T$, we have
\begin{displaymath}
\mathcal{F}(x) = e^{-\xi}\cdot((\overline{x} \cdot e^{-\xi}) * e^{\xi}) \in T \cdot \Big(\big(T \cdot T\big) * T \Big) \subset T \cdot (T * T) \subset T \cdot T \subset T.
\end{displaymath}
Conversely, if $\mathcal{F}(T) \subset T$, then statements $1$ and $2$ of Proposition \ref{propFourierXFonctionXi} allow us to prove that $T$ is stable under Pontryagin product.

\paragraph{$(2) \Leftrightarrow (3)$}
If $\mathcal{F}(T) \subset T$, then we immediately get $\mathcal{F}\mathcal{F}(T) \subset \mathcal{F}(T)$ by applying $\mathcal{F}$.
Conversely, let us assume that $\mathcal{F}\mathcal{F}(T) \subset \mathcal{F}(T)$. By the inversion formula for $\mathcal{F}$, this actually means $T \subset \mathcal{F}(T)$. Applying $\mathcal{F}$ to this inclusion, we get the reversed one, that is statement $(2)$. In particular, we have $(2)$ if and only if we have $(3)$ if and only if  $\mathcal{F}(T) = T$.

\paragraph{$(3) \Leftrightarrow (4)$}
Now we assume that $\mathcal{F}\mathcal{F}(T) \subset \mathcal{F}(T)$ or equivalently $\mathcal{F}(T) = T$. Therefore, since $\xi \in T$ and $T$ is stable by intersection by hypothesis, we have assertion $(4)$ as claimed.
Conversely, assume that $\xi \cdot \mathcal{F}(T) \subset \mathcal{F}(T)$. We want to show that $\mathcal{F}\mathcal{F}(T) \subset \mathcal{F}(T)$. So let us consider a cycle $x \in \mathcal{F}(T)$ and let us use Proposition \ref{propFourierXFonctionXi}, statements $2$ and $4$, together with Corollary \ref{corCalculF(exp(xi))}. These results prove that $\mathcal{F}(x) = e^{-\xi} \cdot ((\overline{x} \cdot e^{-\xi}) * e^{\xi})$ belongs to
\begin{align*}
e^{-\xi} \cdot \Big(\big(\mathcal{F}(T) \cdot e^{-\xi}\big) * \mathcal{F}(T) \Big) \subset e^{-\xi} \cdot \Big(\mathcal{F}(T) * \mathcal{F}(T)\Big) \subset e^{-\xi} \cdot \mathcal{F}(T \cdot T) \subset e^{-\xi} \cdot \mathcal{F}(T) \subset \mathcal{F}(T).
\end{align*}
Hence the claimed inclusion $\mathcal{F}\mathcal{F}(T) \subset \mathcal{F}(T)$; which completes the proof of this proposition.
\end{proof}

\begin{remark}
Note that it is essential in this proof that the polarization $\xi$ belongs to $T$. 
\end{remark}

The two next results will be used several times to exchange pull-backs and push-forwards by isogenies. Indeed it will be very convenient to work with pull-backs (resp. push-forwards) when subalgebras of $\dA(X)$ are endowed with the intersection product (resp. Pontryagin product).

\begin{lemma} \label{lemInversionIsogniePullPush}
Let $\alpha: X \to Y$ be an isogeny between two abelian varieties $X$ and $Y$. There exists an isogeny $\beta: Y \to X$ and an integer $n \in \mathbb{N}^*$ such that $\alpha \circ \beta = n_Y$ and $\beta \circ \alpha = n_X$. Then for all $y \in \dA(Y)$ we have
\begin{align*}
\beta_* y = \frac{1}{\deg (\alpha)} n_{X*} \alpha^* y.
\end{align*}
Therefore, if $y \in \dA^i(Y)_{(s)}$ for some indices $i$ and $s$, then $\beta_* y \in \dA^i(X)_{(s)}$ is proportional to $\alpha^* y$ (and is nonzero if $y \neq 0$).
\end{lemma}

\begin{proof}
Let $y \in \dA(Y)$. Since $\beta \circ \alpha = n_X$ and $\alpha$ is a finite flat morphism of degree $\deg(\alpha) \neq 0$, we have
\begin{displaymath}
n_{X*} \alpha^* y = \beta_* \alpha_* \alpha^* y = \beta_* (\deg(\alpha) y) = \deg(\alpha) \beta_* y,
\end{displaymath}
which means that
\begin{displaymath}
\beta_* y =\frac{1}{\deg(\alpha)} n_{X*} \alpha^* y.
\end{displaymath}
Therefore if $y \in \dA^i(Y)_{(s)}$ it is known that $\alpha^* y \in \dA^i(X)_{(s)}$ is still homogeneous (because $\alpha$ commutes with the multiplication by $n$ on $X$ and $Y$) and so
\begin{displaymath}
\beta_* y =\frac{n^{2 \dim X - 2i +s}}{\deg(\alpha)} \alpha^* y
\end{displaymath}
is proportional to $\alpha^* y$. Finally, as $\beta_*$ and $\alpha^*$ are isomorphisms between $\dA(X)$ et $\dA(Y)$ (because $\alpha$ and $\beta$ are isogenies and we work with algebraic cycles with rational coefficients), $\beta_* y$ is nonzero when $y \neq 0$ (and vice versa).
\end{proof}

\begin{corollary} \label{corInversionIsogeniePullPushAlgebre}
Let $\alpha: X \to Y$ be an isogeny between two abelian varieties $X$ and $Y$. There exists an isogeny $\beta: Y \to X$ and an integer $n \in \mathbb{N}^*$ such that $\alpha \circ \beta = n_Y$ and $\beta \circ \alpha = n_X$. Let $T$ (resp. $T'$) be a bigraded $\mathbb{Q}$-vector subspace of $\dA^\cdot(Y)$ (resp. of $\dA^\cdot(X)$). Then $\alpha^* T = \beta_* T$ and $\beta^* T' = \alpha_* T'$.
\end{corollary}

\begin{proof}
We only prove the first equality $\alpha^* T = \beta_* T$ as the second one can be obtained in a similar way. By hypothesis $T$ is bigraded which means that every $y \in T$ can be (uniquely) written as $y = \sum_{i,s} y_{i,s}$ for some $y_{i,s} \in T^i_{(s)} := T \cap \dA^i(Y)_{(s)}$. The result then follows on from Lemma \ref{lemInversionIsogniePullPush} applied to each $y_{i,s}$:
\begin{displaymath}
\alpha^* y = \sum_{i,s} \alpha^* y_{i,s} = \sum_{i,s} \lambda_{i,s} \beta_* y_{i,s} = \beta_*  \left(\sum_{i,s} \lambda_{i,s} y_{i,s}\right) \in \beta_* T
\end{displaymath}
for some nonzero $\lambda_{i,s} \in \mathbb{Q}$ (if $y_{i,s} = 0$ we can assume that $\lambda_{i,s}  = 1$). Note that we have used here in an essential way that each component $y_{i,s} \in \dA^i(Y)_{(s)}$ defines a class which already belongs to $T$. So we have proven that $\alpha^* T \subset \beta_* T$. The reverse inclusion can be obtained similarly because if $y =  \sum_{i,s} y_{i,s} \in T$ for some $y_{i,s} \in T^i_{(s)}$, then we have
\begin{displaymath}
\beta_* y = \sum_{i,s} \frac{1}{\lambda_{i,s}} \alpha^* y_{i,s} = \alpha^* \left( \sum_{i,s} \frac{1}{\lambda_{i,s}} y_{i,s} \right) \in \alpha^* T.
\end{displaymath}
This shows that $\alpha^* T = \beta_* T$.
\end{proof}

Combining Proposition \ref{propEquivStabFourierPontryaginGeneralisePhi} and Corollary \ref{corInversionIsogeniePullPushAlgebre} with the relation $\varphi_\xi^* \mathcal{F}_X(T) = \psi_{\xi*} \mathcal{F}_X(T)$, we immediately get

\begin{proposition}  \label{propEquivStabFourierPontryaginGeneralisePsi}
Let $T$ be a bigraded $\mathbb{Q}$-subalgebra (for the intersection product) of $\dA(X)$. We suppose that $T$ contains the class of the polarization $\xi$ on $X$. The following statements are equivalent:
\begin{enumerate}
\item $T * T \subset T$.
\item $\psi_{\xi*} \mathcal{F}_X(T) \subset  T$.
\item $\psi_{\xi*} \mathcal{F}_X \psi_{\xi*} \mathcal{F}_X(T) \subset \psi_{\xi*} \mathcal{F}_X(T)$.
\item $\xi \cdot \psi_{\xi*} \mathcal{F}_X(T) \subset \psi_{\xi*} \mathcal{F}_X(T)$.
\end{enumerate}
\end{proposition}

\section{Functoriality of tautological rings $R(C ; J)$} \label{sec:functoriality}

\subsection{Notations} \label{subsec:notations}

In this subsection we present all notations and previous results useful for our work. A more detailed approach of the following notions can be found in \cite[Sections 5.3, 12.1, 12.3]{MR2062673}. Let $C$ and $C'$ be two smooth projective complex curves of genus $g = g(C) \geq 1$ and $g' = g(C') \geq 1$.  We put as always $J = J(C)$ and $J' = J(C')$ for their Jacobians endowed with usual principal polarizations $\Theta$ and $\Theta'$. We avoid the case $g' = 0$ (that is $C' \simeq \mathbb{P}^1$) to spare us some case distinctions when $\dA(J') = \{0\}$. We suppose that we have a finite morphism $f: C \to C'$ of degree $n \in \mathbb{N}^*$. This morphism induces morphisms of abelian varieties:
\begin{align*}
& N_f: J \to J': \mathcal{L}_C \left( \sum n_i P_i \right) \mapsto \mathcal{L}_{C'} \left(\sum n_i f(P_i) \right) \\
& \overline{f} := f^*: J' \to J: \mathcal{L} \mapsto f^* \mathcal{L}.
\end{align*}
Note that $N_f: J \to J'$ is the Albanese morphism induced by $f$ which makes commute the following diagram:
\begin{displaymath}
\xymatrix{
C \ar[r]^f  \ar@{^{(}->}[d]_{f^P} & C' \ar@{^{(}->}[d]^{f^{P'}} \\
J \ar[r]_{N_f} & J'
}
\end{displaymath}
where $P$ is any fixed rational point on $C$ and $P' := f(P) \in C'$. In particular, as $C$ and $C'$ generate $J$ and $J'$ respectively (as abelian varieties), the surjectivity of $f$ implies the surjectivity of $N_f$.

\vspace*{5pt}

Denote by $Y := \im(\overline{f}) \subset J$  (see \cite{MR2062673}). The map $\overline{f}$ factors through an isogeny $j: J' \to Y$ followed by the canonical embedding $\iota_Y: Y \hookrightarrow J$. Also consider the polarization $\varphi_{\iota_Y^* \Theta}$ on $Y$ (a priori non principal) induced by $\Theta$.
Denote also by $e(Y)$ the exponent of $Y$, that is the exponent of the finite group $\ker \varphi_{\iota_Y^* \Theta}$. It is known (see for example \cite[Proposition $1.2.6$]{MR2062673}) that the map
\begin{displaymath}
\psi_{\iota_Y^* \Theta} := e(Y) \varphi_{\iota_Y^* \Theta}^{-1}: \widehat{Y} \to Y \in \hom(\widehat{Y},Y) \otimes \mathbb{Q}
\end{displaymath}
is a morphism (that is it belongs to $\hom(\widehat{Y},Y)$) and even an isogeny.
Consider the following elements
\begin{displaymath}
N_Y := \iota_Y \psi_{\iota_Y^* \Theta} \widehat{\iota_Y} \varphi_\Theta \in \End(J) \qquad \text{and} \qquad \varepsilon_Y := \iota_Y \varphi_{\iota_Y^* \Theta}^{-1} \widehat{\iota_Y} \varphi_\Theta \in \End^0(J) := \End(J) \otimes \mathbb{Q}.
\end{displaymath}
By definition, we have $N_Y = e(Y)\varepsilon_Y$.
Denote by $R: \End^0(J) \to \End^0(J)$ the Rosati involution on $J$ with respect to the Theta polarization and defined by $R(f) := \varphi_{\Theta}^{-1} \circ \widehat{f} \circ \varphi_{\Theta}$.
According to \cite[Lemma 5.3.1]{MR2062673} we have $R(N_Y) = N_Y$ and $N_Y^2 = e(Y) N_Y$. This implies immediately that $R(\varepsilon_Y) = \varepsilon_Y$ and $\varepsilon_Y^2 = \varepsilon_Y$. In other words, $N_Y$ is symmetric and $\varepsilon_Y$ is a symmetric idempotent element of $\End^0(J)$. 
Note that these morphisms are (in particular) linked by the following facts:
\begin{enumerate}
\item[(1)] $\widehat{N_f} = \overline{f}$ after identifying Jacobians and duals \cite[Equation (2) p331]{MR2062673},
\item[(2)] $N_f \overline{f} = n \cdot \id_{J'}$ by definition of $N_f$ and $\overline{f}$, 
\item[(3)] $\overline{f} N_f = \frac{n}{e(Y)} N_Y$ \cite[Proposition 12.3.2]{MR2062673} and in particular, since $\frac{n}{e(Y)} N_Y = n \varepsilon_Y \in \End(J)$, we deduce that $e(Y)$ divides $n$ thanks to \cite[Proposition 12.1.1]{MR2062673},
\item[(4)] $N_{Y|Y} = e(Y) \cdot \id_Y$ \cite[p125]{MR2062673},
\item[(5)] $Y = \im(\overline{f}) = \im(\overline{f}N_f) = \im(N_Y)$ is isogenous to $J'$.
\end{enumerate}

Besides, the map $Y \mapsto \varepsilon_Y$ defines a bijection between the set of abelian subvarieties of $J$ and symmetric idempotents in $\End^0(J)$ \cite[Theorem 5.3.2]{MR2062673}. This yields a natural subvariety of $J$, denoted by $Z$, which is complementary to $Y$ (with respect to the Theta polarization on $J$). This subvariety is associated to the symmetric idempotent element $1 - \varepsilon_Y$ and satisfies
\begin{displaymath}
Z = \im(N_Z) = \ker(N_Y)^0 = \ker(N_f)^0 = \ker(\widehat{\iota_Y}) \simeq \widehat{J/Y}
\end{displaymath}
where $N_Z$ is the norm-endomorphism of $J$ associated to $Z$. It is defined similarly to $N_Y$.

\vspace*{5pt}

Since $(J,\Theta)$ is principally polarized, the complementary subvarieties $Y$ and $Z$ have same exponent \cite[Corollary 12.1.2]{MR2062673}.
Finally, let us recall the following relations \cite[p125]{MR2062673}
\begin{displaymath}
N_{Y|Z} = 0 \qquad \text{and} \qquad N_Y N_Z = 0 \qquad \text{and} \qquad N_Y + N_Z = e(Y) \cdot \id_J.
\end{displaymath}
This provides an isogeny $\mu := \iota_Y + \iota_Z: Y \times Z \to J$ \cite[Corollary 5.3.6]{MR2062673}.

\vspace*{5pt}

At this point, it is useful to look at the commutative diagram
\begin{small}
\begin{displaymath}
\xymatrix{
& J' \ar[rr]^j \ar[dl]_n  \ar[dd]^{\varphi_{\overline{f}^* \Theta}} \ar@/^1.7pc/[rrrr]^{\overline{f}} & & Y \ar@{^{(}->}[rr]^{\iota_Y} \ar[dd]_{\varphi_{\iota_Y^* \Theta}} \ar[dr]^{e(Y)} &  & J \ar[dd]^{\varphi_\Theta}  \ar@/_0.75pc/[llllld]^{N_f} \\
J' \ar[dr]_{\varphi_{\Theta'}} &   & &  & Y & \\
& \widehat{J'} & & \widehat{Y} \ar[ll]^{\widehat{j}} \ar[ru]_{\psi_{\iota_Y^* \Theta}}   &                                             & \widehat{J}. \ar[ll]^{\widehat{\iota_Y}} \ar@/^1.7pc/[llll]^{\widehat{\overline{f}}}
}
\end{displaymath}
\end{small}
Note that commutativity is justified by the identities $(1) - (5)$ recalled above and following relations:
\begin{enumerate}
\item[(6)] 
$\varphi_{\iota_Y^* \Theta} = \widehat{\iota_Y} \varphi_\Theta \iota_Y$. This can be checked immediately on points $y \in Y$:
\begin{displaymath}
\varphi_{\iota_Y^* \Theta}(y) := t_y^* \iota^*_Y \mathcal{L}_J(\Theta) \otimes \iota_Y^* \mathcal{L}_J(\Theta)^\vee = \iota_Y^* \left( t^*_{\iota_Y(y)} \mathcal{L}_J(\Theta) \otimes \mathcal{L}_J(\Theta)^\vee \right) =: \iota_Y^* \varphi_\Theta(\iota_Y(y)).
\end{displaymath}
\item[(7)] In the same way, $\varphi_{\overline{f}^* \Theta} = \widehat{\overline{f}} \varphi_{\Theta} \overline{f}$.
\item[(8)] Lemma 12.3.1 of \cite{MR2062673} states that $\varphi_{\overline{f}^* \Theta} = \varphi_{n \Theta'} = n \varphi_{\Theta'}$.
\end{enumerate}

We now have all necessary tools to study the functoriality of tautological rings $R(C ; J)$.

\subsection{Functoriality of tautological rings $R(C ; J)$}

Let us start with a very simple proposition which is the key to all following results in this section.

\begin{proposition} \label{propPushForwardNfFbarre}
Let $f: C \to C'$ be a finite morphism of degree $n$. We have
\begin{displaymath}
(N_f)_* C = n C' 
\qquad \text{and} \qquad
\overline{f}_* C' = \frac{1}{n} \left(\frac{n}{e(Y)} N_Y \right)_* C = \frac{1}{n} \varepsilon_{Y*} n_* C.
\end{displaymath}
\end{proposition}

\begin{proof}
The first formula is clear and by combining it with relation (3) we can derive the second one as follow:
\begin{align*}
n \overline{f}_* C' = \overline{f}_* (N_f)_* C = (\overline{f}N_f)_* C = \left(\frac{n}{e(Y)} N_Y \right)_* C.
\end{align*}
\end{proof}

This proposition immediately implies:

\begin{corollary} \label{corPushForwardNfCi}
Let $f: C \to C'$ be a finite morphism of degree $n$. For all $i \in \llbracket 0,g-1\rrbracket$ we have
\begin{displaymath}
(N_f)_* C_{(i)} = n C'_{(i)} \in \dA^{g'-1}(J')_{(i)}.
\end{displaymath}
\end{corollary}

\begin{proof}
Decomposing $C = C_{(0)} + \ldots + C_{(g-1)}$ and $C' = C'_{(0)} + \ldots + C'_{(g'-1)}$, the equality $(N_f)_* C = n C'$ gives 
\begin{displaymath}
\sum_{i=0}^{g-1} (N_f)_* C_{(i)} = \sum_{i=0}^{g'-1} n C'_{(i)}.
\end{displaymath}
Since $(N_f)_* C_{(i)} \in \dA^{g'-1}(J')_{(i)}$ \cite[Proposition 2.c]{MR826463}, we have by uniqueness in Beauville's decomposition $(N_f)_* C_{(i)} = n C'_{(i)} \in \dA^{g'-1}(J')_{(i)}$.
\end{proof}

Now we can easily deduce results concerning tautological rings since the cycles $C_{(i)}$ and $C'_{(i)}$ are generators of algebras $R(C ; J)$ and $R(C' ; J')$ for the Pontryagin product.

\begin{corollary} \label{corSurjectiviteNf_*}
Let $f: C \to C'$ be a finite morphism. The map $(N_f)_*$ induces a surjective morphism
\begin{displaymath}
(N_f)_*: R(C ; J) \relbar\joinrel\twoheadrightarrow R(C' ; J').
\end{displaymath}
In particular, $R(C' ; J')$ is a quotient of $R(C ; J)$.
\end{corollary}

\begin{proof}
Since push-forwards are ring morphisms when we consider $\dA(J)$ and $\dA(J')$ endowed with the Pontryagin product,
and since the $C'_{(i)} \in \im ((N_{f})_*)$ generate $R(C' ; J')$ as $\mathbb{Q}$-subalgebra of $\dA(J')$ for the Pontryagin product, the claim follows from Corollary \ref{corPushForwardNfCi}.
\end{proof}

\begin{remark}
$(N_f)_*$ is a surjective morphism as $\mathbb{Q}$-linear map and is also a morphism of $\mathbb{Q}$-algebra when we endow $\dA(J)$ and $\dA(J')$ with the Pontryagin product. Similarly, the next corollary gives a surjective morphism as $\mathbb{Q}$-linear map and also as morphism of $\mathbb{Q}$-algebra for the intersection product.
\end{remark}

By Fourier duality we get the equivalent corollary:

\begin{corollary} \label{corSurjectivitePullBackFbarre}
Let $f: C \to C'$ be a finite morphism. The map $\overline{f}^*$ induces a surjective morphism
\begin{displaymath}
\overline{f}^*: R(C ; J)  \relbar\joinrel\twoheadrightarrow R(C' ; J').
\end{displaymath}
\end{corollary}

\begin{proof}
Let $x \in R(C ; J)$. According to relation $(1)$, we have $\widehat{\overline{f}} = N_f$ (after identifying Jacobians with their duals). Thus we deduce thanks to inversion formulas for the Fourier transforms on $J$ and $J'$ and thanks to Proposition \ref{propCommutationFourierPullBackPushForward} applied to the morphism $\overline{f}: J' \to J$ with $X = J$ and $Y = J'$ that
\begin{align*}
\overline{f}^* x 
& = (-1)^{g'} (-1_{J'})^* \mathcal{F}_{\widehat{J'}} \mathcal{F}_{J'} \overline{f}^* x 
= (-1)^{g'} (-1_{J'})^* \mathcal{F}_{\widehat{J'}} (-1)^{g - g'} \widehat{\overline{f}}_* \mathcal{F}_{J} (x).
\end{align*}
Then keeping the identifications of $J' \simeq \widehat{J'}$ and $J \simeq \widehat{J}$, we get
\begin{displaymath}
\overline{f}^* x  = (-1)^{g} (-1_{J'})^* \mathcal{F}_{J'} (N_f)_* \mathcal{F}_{J}(x) \in R(C' ; J')
\end{displaymath}
because on the one hand $(N_f)_* R(C ; J) \subset R(C' ; J')$ (see Corollary \ref{corSurjectiviteNf_*}) and on the other both $R(C ; J)$ and $R(C' ; J')$ are $\mathbb{Q}$-vector subspaces stable under Fourier transform and under operators $k^*$. This proves the existence of the map $\overline{f}^*: R(C ; J)  \longrightarrow R(C' ; J')$.
The surjectivity of $\overline{f}^*$ follows from the surjectivity of $(N_f)_*$ (Corollary \ref{corSurjectiviteNf_*}). Indeed if $y \in R(C' ; J')$, then there exists an $z \in R(C' ; J')$ such that $y = \mathcal{F}_{J'}(z)$ (by stability of $R(C' ; J')$ under $(-1)^*$, $\mathcal{F}_{J'}$ and inversion formula). Consequently, for some $x \in R(C ; J)$ such that $(N_f)_* x = z$, we still have thanks to Proposition \ref{propCommutationFourierPullBackPushForward}
\begin{displaymath}
y = \mathcal{F}_{J'}(z) = \mathcal{F}_{J'}((N_f)_* x)  = \widehat{N_f}^* \mathcal{F}_J(x) = \overline{f}^* \mathcal{F}_J(x)
\in \overline{f}^* R(C ; J)
\end{displaymath}
because $R(C ; J)$ is stable under $\mathcal{F}_J$.
\end{proof}

Now we would like to consider, roughly, $R(C' ; J')$ from the point of view of $\dA(J)$. That is we are interested in the rings $\overline{f}_* R(C' ; J') \subset \dA(J)$ and $(N_f)^* R(C' ; J') \subset \dA(J)$. The intuition suggests that cycles in $\overline{f}_* R(C' ; J')$ and $(N_f)^* R(C' ; J')$ should be with support on $Y$ (recall that $Y$ is the subvariety of $J$ isogenous to $J'$). The next two results explain this fact.

\begin{proposition} \label{propImageR(C')dansY}
Let $f: C \to C'$ be a finite morphism. The isogeny $j: J' \to Y$, corestriction map of $\overline{f} = f^*$, induces an isomorphism
\begin{displaymath}
j_*: R(C' ; J')  \overset{\simeq}{\longrightarrow} j_* R(C' ; J') = \iota_Y^* R(C ; J) \subset \dA(Y).
\end{displaymath}
\end{proposition}

\begin{proof}
The morphism $j: J' \to Y$ is an isogeny (in particular, it is finite and flat). Therefore
\begin{displaymath}
j_* j^* = \deg(j) \cdot \id_{\dA(Y)}: \dA(Y) \to \dA(Y).
\end{displaymath}
So applying $j_*$ to the relation of the previous corollary, we deduce
\begin{align*}
j_* R(C' ; J') =  j_* \overline{f}^* R(C ; J)  
= j_* (\iota_Y \circ j)^* R(C ; J) = j_* j^* \iota_Y^* R(C ; J) = \deg(j) \iota_Y^* R(C ; J) = \iota_Y^* R(C ; J).
\end{align*}
Moreover, as $j$ is an isogeny, there exists an isogeny $h: Y \to J'$ such that $h \circ j = d_{J'}$ and $j \circ h = d_{Y}$ for some $d \in \mathbb{N}^*$. In particular, $j_* \left(\frac{1}{d}h\right)_* = \id_{\dA(Y)}$ and $\left(\frac{1}{d}h\right)_* j_* = \id_{\dA(J')}$, so that $j_*$ is an isomorphism.
\end{proof}

\begin{corollary} \label{corImageR(C')dansYfbarre}
Let $f: C \to C'$ be a finite morphism. The map $\overline{f}: J' \to J$ induces a surjective morphism
\begin{displaymath}
\overline{f}_*: R(C' ; J')  \relbar\joinrel\twoheadrightarrow \iota_{Y*} \iota_Y^* R(C ; J) = [Y] \cdot R(C ; J) \subset \dA(J).
\end{displaymath}
By Fourier duality we obtain similarly a surjective morphism
\begin{displaymath}
N_f^*: R(C' ; J') \relbar\joinrel\twoheadrightarrow \psi_Y^* \psi_{Y*} R(C ; J) = \varphi_{\Theta}^* \mathcal{F}_J([Y]) * R(C ; J)
\end{displaymath}
with $\psi_Y := \psi_{\iota_Y^* \Theta} \circ \widehat{\iota_Y} \circ \varphi_\Theta \in \hom(J,Y)$.
\end{corollary}

\begin{proof}
The first assertion is a direct consequence of Proposition \ref{propImageR(C')dansY} because $\overline{f}_* = \iota_{Y*} \circ j_*$. See also \cite[Example 8.1.1]{MR1644323}.
The second statement can be deduced from the first one by using Proposition \ref{propCommutationFourierPullBackPushForward} and the fact that Fourier transforms on $J$ and $J'$ respectively induce automorphisms of $R(C ; J)$ and $R(C' ; J')$. Indeed, recall \cite{MR2041776} that $\varphi_{\Theta'}^* \mathcal{F}_{J'}( R(C' ; J')) = R(C' ; J')$ and similarly $\varphi_{\Theta}^* \mathcal{F}_{J}( R(C ; J)) = R(C ; J)$. Then, we have on the one hand
\begin{align*}
\varphi_{\Theta}^* \mathcal{F}_J( \overline{f}_* R(C' ; J'))
& = \varphi_{\Theta}^* \widehat{\overline{f}}^* \mathcal{F}_{J'}( R(C' ; J'))
= \varphi_{\Theta}^* \widehat{\overline{f}}^* \varphi_{\Theta'}^{-1*} \varphi_{\Theta' }^{*}\mathcal{F}_{J'}( R(C' ; J')) 
= N_f^* R(C' ; J').
\end{align*}
And on the other hand,
\begin{align*}
\varphi_{\Theta}^* \mathcal{F}_J( \iota_{Y*} \iota_Y^* R(C ; J))
& = (-1)^{g-g'}\varphi_{\Theta}^* \widehat{\iota_Y}^* \widehat{\iota_Y}_*  \mathcal{F}_J(R(C ; J))
= \varphi_{\Theta}^* \widehat{\iota_Y}^* \widehat{\iota_Y}_* \varphi_{\Theta *} \varphi_{\Theta}^* \mathcal{F}_J(R(C ; J)) \\
& = \varphi_{\Theta}^* \widehat{\iota_Y}^* \widehat{\iota_Y}_* \varphi_{\Theta *} R(C ; J).
\end{align*}
Besides, we have
\begin{align*}
\varphi_{\iota_Y^* \Theta *}\varphi_{\iota_Y^* \Theta}^* \widehat{\iota_Y}_* \varphi_{\Theta *} R(C ; J) 
= \deg(\varphi_{\iota_Y^* \Theta})\widehat{\iota_Y}_* \varphi_{\Theta *} R(C ; J)
= \widehat{\iota_Y}_* \varphi_{\Theta *} R(C ; J)
\end{align*}
and using Corollary \ref{corInversionIsogeniePullPushAlgebre} twice, we get
\begin{displaymath}
\psi_{\iota_Y^* \Theta}^* \psi_{\iota_Y^* \Theta *}\widehat{\iota_Y}_* \varphi_{\Theta *} R(C ; J) = \widehat{\iota_Y}_* \varphi_{\Theta *} R(C ; J).
\end{displaymath}
Therefore we obtain
\begin{displaymath}
N_f^* R(C' ; J') = \varphi_{\Theta}^* \widehat{\iota_Y}^* \psi_{\iota_Y^* \Theta}^* \psi_{\iota_Y^* \Theta *}\widehat{\iota_Y}_* \varphi_{\Theta *} R(C ; J) = \psi_Y^* \psi_{Y*} R(C ; J).
\end{displaymath}
Finally, the last assertion follows from the equalities (obtained thanks to Proposition \ref{propFourierXFonctionXi})
\begin{displaymath}
\varphi_\Theta^* \mathcal{F}_J([Y] \cdot R(C ; J)) = (-1)^g \varphi_\Theta^* \mathcal{F}_J([Y]) * \varphi_\Theta^* \mathcal{F}_J(R(C ; J)) = \varphi_\Theta^* \mathcal{F}_J([Y]) * R(C ; J).
\end{displaymath}
\end{proof}

\subsection{The special case of $n$-cyclic Galois coverings}

In this section we get more explicit results when the covering $f: C \to C'$ is associated to an automorphism of the curve $C$. We start with definitions.

\begin{definition}
A \emph{finite Galois covering} is a finite morphism $f: C \to C'$ of smooth projective complex curves $C$ and $C'$ such that there is an isomorphism $C' \simeq C/ \aut(f)$ where
\begin{displaymath}
\aut(f) := \{ \mu \in \aut(C) ~|~f \circ \mu = f\}
\end{displaymath}
denotes the automorphism group of the Galois covering.
This amounts to say that the function field extension $\dK(C)/\dK(C')$ is Galois.
\end{definition}

The function field $\dK(C')$ is then given by the subfield of invariants $\dK(C)^{\aut(f)} \subset \dK(C)$ according to the Galois group $\gal(\dK(C)/\dK(C')) \simeq \aut(f)$.

\begin{definition}
Let $f: C \to C'$ be a Galois covering of smooth projective complex curves. We say that $f$ is a \emph{$n$-cyclic Galois covering} if $\aut(f) \simeq \mathbb{Z}/n\mathbb{Z}$. In that case, we  will usually consider a generator $\sigma \in \aut(f)$ so that $C' \simeq C/\langle\sigma\rangle$.
\end{definition}

We start with a lemma which specifies general facts concerning the subvariety $Y$ and the automorphism $\sigma$ defining a cyclic Galois covering $f : C \to C/\langle\sigma\rangle$.

\begin{lemma} \label{lemExposantYGeneralise}
Let $f: C \to C' \simeq C/\langle\sigma\rangle$ be an $n$-cyclic Galois covering associated to an automorphism $\sigma \in \aut(C)$ of order $n \in \mathbb{N}^*$ (with possibly $g(C') = 0$). Then
\begin{enumerate}
\item $\overline{f}: J' \to J$ induces an isogeny $j: J' \to Y := \im(\overline{f}) \subset J$ of degree dividing $n$. Furthermore this isogeny is an isomorphism if and only if $f$ does not factorize via a cyclic étale covering $f': C'' \to C'$ of degree $\geq 2$.
\item $\overline{f} N_f = \Phi_n(\sigma)$ with $\Phi_n(X) = 1 + X + \ldots + X^{n-1}$. Therefore $N_Y = \frac{e(Y)}{n}\Phi_n(\sigma)$.
\item $Y = \ker(\sigma -1)^0$.
\item We have the equality $e(Y) = 1$ if and only if $Y = J(C)$ or $Y = \{0\}$ if and only if $n= 1$ or $C' \simeq \mathbb{P}^1$.
\end{enumerate}
\end{lemma}

\begin{proof}
~
\begin{enumerate}
\item According to \cite[Proposition 11.4.3]{MR2062673}, $j$ is an isomorphism (that is to say $\overline{f}$ is injective) if and only if $f$ does not factorize via a cyclic étale covering of degree $\geq 2$. More precisely (see \cite[Corollary 11.4.4]{MR2062673}), when $\overline{f}$ is non injective, $f$ factorizes via a cyclic étale covering $f_e : C_e \to C'$ of degree $\geq 2$ and such that $\deg(j) := \# \ker(\overline{f}) = \# \ker(f_e^*)$. Since $f_e$ is a cyclic étale covering, we also have $\# \ker(f_e^*) = \deg(f_e)$, which divides $n = \deg(f)$ by multiplicativity of the degree map.
\item The fibres of $f: C \to C'$ are cyclic orbits for the action of $\langle \sigma \rangle$ on $C$ (because $f$ is Galois). Then for every point $z \in J$ represented by $\mathcal{L}_C(\sum n_i P_i)$, we have thanks to Relation $(3)$ of Section \ref{subsec:notations}:
\begin{align*}
\frac{n}{e(Y)} N_Y(z) & = \overline{f} N_f \left( \mathcal{L}_C(\sum n_i P_i) \right) = \overline{f} \left( \mathcal{L}_{C'}(\sum n_i f(P_i)) \right) \\
& = \mathcal{L}_C \left( \sum n_i ( P_i + \sigma(P_i) + \ldots + \sigma^{n-1}(P_i) ) \right)  \\
& = z + \sigma(z) + \ldots + \sigma^{n-1}(z) = \Phi_n(\sigma)(z).
\end{align*}
So $\frac{n}{e(Y)} N_Y = \Phi_n(\sigma)$ that is $N_Y = \frac{e(Y)}{n} \Phi_n(\sigma)$.
\item We now have to justify the equality $Y := \im(f^*) = \ker(\sigma -1)^0$. In order to do this, let us begin by noting that a point $x \in J$ (corresponding to a class of invertible sheaf $\mathcal{L} \in \pic^0(C)$) belongs to $\ker(\sigma -1)$ if and only if $\sigma^* \mathcal{L} \simeq \mathcal{L}$. Indeed, the Albanese morphism $\sigma : J \to J$ is the inverse map of $\overline{\sigma} = \sigma^*: \pic^0(C) \to \pic^0(C)$ (thanks to relation \ref{subsec:notations} (2)). Then we have
\begin{align*}
\sigma(x) = x 
\quad \Longleftrightarrow \quad \sigma^{-1}(x) = x
\quad \Longleftrightarrow \quad \sigma^* \mathcal{L} \simeq \mathcal{L}.
\end{align*}
Moreover, $f \circ \sigma = f$ (by definition of the quotient $C/\langle\sigma\rangle$). Then each element $\mathcal{L} := f^* \mathcal{M} \in \im(f^*)$ is invariant under $\sigma^*$. Indeed, one has $\sigma^* \mathcal{L} \simeq \sigma^* f^* \mathcal{M} \simeq f^* \mathcal{M} \simeq \mathcal{L}$.
Thus we have proven that $\im(f^*) \subset \ker(\sigma -1)$ and by the connectedness of $\im(f^*)$, we even obtain $\im(f^*) \subset \ker(\sigma-1)^0$. This leads us to following inclusions (using assertion (2))
\begin{align*}
Y \subset \ker(\sigma -1)^0 \subset \ker(e(Y) - N_Y)^0  = \ker(N_Z)^0.
\end{align*}
But we know that $Y = \ker(N_Z)^0$, which can be proven by the following argument
\begin{align*}
\dim \ker(N_Z)^0 & = \dim J(C) - \dim \im(N_Z)  = \dim J(C) - \dim Z \\
& = g(C) - (\dim J(C) - \dim Y) = g(C) - g(C) + \dim Y = \dim Y.
\end{align*}
Hence the previous inclusions are in fact equalities:
\begin{displaymath}
Y = \ker(\sigma -1)^0 = \ker(e(Y) - N_Y)^0 = \ker(N_Z)^0.
\end{displaymath}
\item As $(J,\Theta)$ is a principally polarized abelian variety, $Y$ and $Z = \im(e(Y) - N_Y)$ have same exponent $e(Y) = e(Z)$ (see Subsection \ref{subsec:notations} or more directly \cite[Corollary 12.1.2]{MR2062673}). If this exponent is equal to $1$, then the polarizations induced by $\Theta$ on $Y$ and $Z$, namely $\varphi_{\iota_Y^* \Theta}$ and $\varphi_{\iota_Z^* \Theta}$, are principal polarizations.
So Lemma 12.1.6 of \cite{MR2062673} implies that the isogeny 
\begin{displaymath}
\mu := \iota_Y + \iota_Z: (Y \times Z, \mu^*\Theta = p_Y^* \iota_Y^* \Theta + p_Z^* \iota_Z^* \Theta) \to (J,\Theta)
\end{displaymath}
which is of degree $\#(Y \cap Z) = \#  \ker(\varphi_{\iota_Y^* \Theta}) = 1$ is an isomorphism of principally polarized abelian varieties. Since $\Theta$ is irreducible, we have $Y = \{0\}$ or $Z = \{0\}$. The first case means that $C' \simeq C/\langle\sigma\rangle \simeq \mathbb{P}^1$ because $Y$ is isogenous to $J(C/\langle\sigma\rangle)$. The second case means that $J = Y = \ker(\sigma-1)^0$ (according to assertion (3)); that is $\sigma = 1$.
\end{enumerate}
\end{proof}

\begin{remark}
The dimension argument used to prove assertion (3) of this lemma can be replaced by the construction of a section to the inclusion $Y \hookrightarrow \ker(\sigma -1)^0$. This can be achieved thanks to a descent lemma (see \cite[Théorème 2.3]{MR999313}).
\end{remark}

The next (easy) lemma will be widely used in the sequel. It is a special case of \cite[Proposition 11.5.3]{MR2062673} for the correspondence $\Gamma_\sigma$ on $C \times C$, namely the graph of $\sigma$.

\begin{lemma}  \label{lemRosatiSigma}
Let $\sigma \in \aut(C)$ be an automorphism of $C$. As before, we denote by $\sigma$ the Albanese automorphism induced in $\End(J)$ and $R$ the Rosati involution on $\End^0(J)$ (with respect to the Theta polarization). Then $R(\sigma) = \sigma^{-1}$. Accordingly, we have for all $P \in \mathbb{Q}[X]$
\begin{displaymath}
R(P(\sigma)) = P(\sigma^{-1}). 
\end{displaymath}
\end{lemma}

\begin{proposition} \label{propPushForwardf_*C'}
Let $f: C \to C' \simeq C/\langle\sigma\rangle$ be an $n$-cyclic Galois covering associated to an automorphism $\sigma \in \aut(C)$ of order $n \in \mathbb{N}^*$. The map $\overline{f}$ induces a surjective morphism
\begin{displaymath}
\overline{f}_*: R(C' ; J')  \relbar\joinrel\twoheadrightarrow  \Phi_n(\sigma)_* R(C ; J).
\end{displaymath}
More precisely, the following equality holds
\begin{displaymath}
\overline{f}_* C' = \frac{1}{n} \Phi_n(\sigma)_* C.
\end{displaymath}
Likewise, $N_f$ induces a surjective morphism
\begin{displaymath}
N_f^*: R(C' ; J') \relbar\joinrel\twoheadrightarrow \Phi_n(\sigma)^* R(C ; J).
\end{displaymath}
\end{proposition}

\begin{proof}
Recall that $\overline{f} N_f = \Phi_n(\sigma)$ by Lemma \ref{lemExposantYGeneralise} (2). Consider a cycle $y \in R(C' ; J')$. By Corollary \ref{corSurjectiviteNf_*} $(N_f)_*: R(C ; J) \to R(C' ; J')$ is surjective so there exists $x \in R(C ; J)$ such that $(N_f)_* x = y$. Hence
\begin{displaymath}
\overline{f}_* y =\overline{f}_* (N_f)_* x = (\overline{f} N_f)_* x =  \Phi_n(\sigma)_* x \in \Phi_n(\sigma)_* R(C ; J).
\end{displaymath}
Conversely, for all $x \in R(C ; J)$,
\begin{displaymath}
\Phi_n(\sigma)_* x = \overline{f}_* (N_f)_* x = \overline{f}_* y \in \overline{f}_* R(C ; J)
\end{displaymath}
where $y := (N_f)_* x \in R(C' ; J')$.
Using $C' = \frac{1}{n}(N_f)_*C$ (Proposition \ref{propPushForwardNfFbarre}), we obtain
\begin{displaymath}
\overline{f}_* C' =  \frac{1}{n} \overline{f}_* (N_f)_* C = \frac{1}{n} \Phi_n(\sigma)_* C.
\end{displaymath}
Then note that the Rosati involution fixes $\Phi_n(\sigma)$. Indeed, according to Lemma \ref{lemRosatiSigma}, we have
\begin{displaymath}
R(\Phi_n(\sigma)) = \Phi_n(\sigma^{-1}) = \Phi_n(\sigma) \in \End(J).
\end{displaymath}
To get the second statement about $N_f^*$ and thus conclude the proof, it remains to use this fact, Proposition \ref{propCommutationFourierPullBackPushForward}, assertion $(1)$ of Subsection \ref{subsec:notations} and the fact that Fourier transforms on $J$ and $J'$ induce automorphisms of $R(C ; J)$ and $R(C' ; J')$ respectively.
\end{proof}

We can derive from this proposition the following corollary (as we obtained Corollary  \ref{corPushForwardNfCi}).

\begin{corollary}
Let $f: C \to C' \simeq C/\langle\sigma\rangle$ be an $n$-cyclic Galois covering associated to an automorphism $\sigma \in \aut(C)$ of order $n \in \mathbb{N}^*$. For all indices $i \in \llbracket 0,g' -1 \rrbracket$,  we have
\begin{displaymath}
\overline{f}_* C'_{(i)} = \frac{1}{n} \Phi_n(\sigma)_* C_{(i)} \in \dA^{g-1}(J)_{(i)}.
\end{displaymath}
\end{corollary}

At this point, we would like to stress that push-forwards by polynomials in the automorphism appear naturally when considering tautological rings associated to curves with automorphisms. This may be the main idea to keep in mind about this whole section on Galois coverings. It raises the question to get a better understanding of cycle classes of the form $P(\sigma)_* C$ and motivates the study of the tautological ring containing all of them. This is the purpose of the rest of this paper.

\section{The tautological ring $R_G(C ; J)$} \label{sec:ringJacobian}

Let $C$ be a smooth projective complex curve of genus $g \geq 1$. Until the end of this section we assume that we have a finite automorphism group $G \subset \aut(C)$. We use the same notation $G$ for the corresponding subgroup of $\aut(J)$ and we shall note by $\mathbb{Z}[G]$ the subring of $(\End(J),+,\circ)$ formed by polynomials in elements of $G$, that is the image in $\End(J)$ of the group ring $\mathbb{Z}[G]$. Note that if $G$ is an abelian group generated by automorphisms $\sigma_1,\ldots,\sigma_s$ of finite order, then $\mathbb{Z}[G]$ identifies with $\mathbb{Z}[\sigma_1,\ldots,\sigma_s] \subset \End(J)$. 

\begin{remark}
Recall  that if $g \geq 2$, then any $\sigma \in \aut(C)$ is finite.
\end{remark}

Now we want to prove Theorem \ref{theoRsigmiG} which provides a set of generators for the tautological ring
\begin{displaymath}
R_G(C ; J) := \taut_J \Big(\{\pi_* C \in \dA(J) ~|~ \pi \in  \mathbb{Z}[G] \} \Big).
\end{displaymath}

The main difficulty is to show that the algebra for the intersection product generated by cycles of the form $\pi^* N^i(w)$ is stable under Pontryagin product too. Thus we first prove the following:

\subsection{Key-theorem}

\begin{theorem} \label{theoSsigmaiStablePontryagin}
Let $S_G := S_G(C ; J) \subset \dA(J)$ be the $\mathbb{Q}$-subalgebra (for the intersection product)  generated by the $\pi^* N^i(w)$ for $\pi \in \mathbb{Z}[G]$ and $i \in \llbracket 1,g-1 \rrbracket$. Then $S_G$ is stable under the Pontryagin product.
\end{theorem}

To prove this theorem we will use Beauville's strategy \cite{MR2041776} which essentially consists in using the Fourier transform on $J$ and more specifically we will use implication $(4) \Rightarrow (1)$ of Proposition \ref{propEquivStabFourierPontryaginGeneralisePhi}. To be brief, we will denote by $\mathcal{F}$ the automorphism $\varphi_\Theta^* \mathcal{F}_J: \dA(J) \to \dA(J)$. We always identify $J$ and $\widehat{J}$ via the principal polarization $\varphi_{\Theta}$. In particular, we will consider the Poincaré line bundle on $J \times J$, namely $\mathcal{P}_{J \times J} := (1 \times \varphi_{\Theta})^* \mathcal{P}_{J \times \widehat{J}}$, and its cycle class $l_{J \times J} = m^* \theta - p^* \theta - q^* \theta \in \dA^1(J\times J)$.

\begin{remark} \label{remConventionFp}
As always we fix a rational point $P$ on $C$ to embed the curve $C$ in $J = J(C)$. We then recall some relations whose proofs can be found in \cite{MR861976} (see Summary 6.11).
\begin{enumerate}
\item We put $\mathcal{L}^P := \mathcal{L}(\Delta_C - P \times C - C \times P) \in \pic(C^2)$.
\item There is an invertible sheaf $\mathcal{M}^P \in \pic(C \times J)$ such that $(1 \times f^P)^* \mathcal{M}^P \simeq \mathcal{L}^P$.
\item $\mathcal{M}^P \simeq (f^P \times (-1))^* (1 \times \varphi_{\Theta})^* \mathcal{P}_{J \times \widehat{J}} \simeq (f^P \times (-1))^* \mathcal{P}_{J \times J} \simeq (f^P \times 1)^* \mathcal{P}_{J \times J}^\vee$.
\item $\mathcal{L}^P \simeq (f^P \times f^P)^* \mathcal{P}_{J \times J}^\vee \simeq (f^P \times f^P)^* (p^* \mathcal{L}_J(\Theta) \otimes q^* \mathcal{L}_J(\Theta) \otimes m^* \mathcal{L}_J(\Theta)^\vee)$.
\item There is a map $f^{P \vee}: \widehat{J} \to J$ such that $(f^P \times 1)^* \mathcal{P}_{J \times \widehat{J}} \simeq (1 \times f^{P\vee})^* \mathcal{M}^P$. On points, $f^{P \vee}$ is induced by $f^{P*} : \pic(J) \to \pic(C)$. We have $f^{P\vee} = - \varphi_{\Theta}^{-1}$.
\end{enumerate}
\end{remark}

\vspace*{1pt}

\begin{proof}[Proof of Theorem \ref{theoSsigmaiStablePontryagin}]
We decompose the proof of Theorem \ref{theoSsigmaiStablePontryagin} in several steps.

\paragraph{Step $1$}

By definition $S_G$ is generated as $\mathbb{Q}$-algebra (for the intersection product) by classes $\pi^* N^i(w)$ with $i \in \llbracket 1,g \rrbracket$ and $\pi \in \mathbb{Z}[G]$. Since $\pi^* N^g(w)$ is a multiple of the class of a point and $N^1(w)^g = \theta^g = g! \cdot P$, it suffices to consider indices $i \in \llbracket 1,g-1 \rrbracket$.

Moreover, thanks to Proposition \ref{propCommutationFourierPullBackPushForward} applied with $X = Y = J$ and $\alpha = \pi$, we have
\begin{displaymath}
\mathcal{F}(\pi^* N^i(w)) = R(\pi)_* \mathcal{F}(N^i(w)) = (-1)^{g+i} R(\pi)_* C_{(i-1)}.
\end{displaymath}

Thus $\mathcal{F}(S_G)$ is generated as $\mathbb{Q}$-vector space by products of the form
\begin{displaymath}
(R(\pi_1)_* C_{(i_1-1)}) * (R(\pi_2)_* C_{(i_2-1)}) * \ldots * (R(\pi_{r})_* C_{(i_r-1)}).
\end{displaymath}
By Lemma \ref{lemRosatiSigma}, we get that each $R(\pi) \in \mathbb{Z}[G]$. Precisely, if $\pi$ is a finite sum $\pi = \sum_{g \in G} a_g\circ g$ with coefficients $a_g \in \mathbb{Z}$, then $R(\pi) = \sum_{g \in G} a_g \circ g^{-1}$. In other words, the Rosati involution induces an involution of $\mathbb{Z}[G]$. Consequently, $\mathcal{F}(S_G)$ is generated as $\mathbb{Q}$-vector space by products
\begin{displaymath}
(\pi_{1*} C_{(i_1-1)}) * (\pi_{2*} C_{(i_2-1)}) * \ldots * (\pi_{r*} C_{(i_r-1)})
\end{displaymath}
for $\pi_j \in \mathbb{Z}[G]$ and integers $i_j \in \llbracket 1,g-1\rrbracket$. We now obtain a more convenient set of generators for $\mathcal{F}(S_G)$ thanks to the following lemma which is inspired by \cite[Lemma $4.2$]{MR2041776} and is proven in the same way.

\begin{lemma} \label{lemGenFS''casRevetGeneral}
$\mathcal{F}(S_G)$ is generated as $\mathbb{Q}$-vector space by the classes of the form $(k_{1*} \pi_{1*} C) * \ldots * (k_{r*} \pi_{r*} C)$ for all sequences $(k_1,\ldots,k_r) \in (\mathbb{N}^*)^r$ and all $\pi_j \in \mathbb{Z}[G]$. 
Therefore it is generated as $\mathbb{Q}$-vector space by classes of the form $(\pi_{1*} C) * \ldots * (\pi_{r*} C)$ for $\pi_j \in \mathbb{Z}[G]$.
\end{lemma}

\paragraph{Step $2$}

According to Proposition \ref{propEquivStabFourierPontryaginGeneralisePhi} (that we can apply since $\theta = N^1(w) \in  S_G$), it remains essentially to prove that $\theta \cdot \mathcal{F}(S_G) \subset \mathcal{F}(S_G)$. Actually we will show that for all nonzero $\pi_1,\ldots,\pi_r \in \mathbb{Z}[G]$ the class $\theta \cdot \left[ (\pi_{1*} C) * \ldots * (\pi_{r*} C) \right]$ belongs to $\mathcal{F}(S_G)$.

If $r = 0$, we have $\theta = N^1(w) \in R(C; J) = \mathcal{F}(R(C ; J)) \subset \mathcal{F}(S_G)$.
Also note that if $r = 1$, $\theta \cdot \pi_{1*} C \in \dA^g(J)$ is a multiple of the class of a point. Therefore it is a multiple of $\mathcal{F}([J]) \in \mathcal{F}(S_G)$. Thus we suppose from now on that  $r \geq 2$ in which case we consider the following map:
\begin{displaymath}
u: C^r \overset{\Phi}{\longrightarrow} J^r \overset{K}{\longrightarrow} J^r \overset{m}{\longrightarrow} J
\end{displaymath}
with $\Phi := f^P \times \ldots \times f^P$ ($r$ times), $K := \pi_1 \times \ldots \times \pi_r$ and where $m : J^{r} \to J$ is induced by the multiplication on $J$. Then the cycle $\theta \cdot \left[ (\pi_{1*} C) * \ldots * (\pi_{r*} C) \right]$ is a multiple of $u_* u^* \theta$ (by the projection formula). We now introduce projections $p_i: J^r \to J$ and $p_{ij}: J^r \to J^2$. In the same way we consider projections $q_i: C^r \to C$ and $q_{ij}: C^r \to C^2$. A calculation similar to that in \cite[Section 4.3]{MR2041776} yields that $u_* u^* \theta$ is a linear combination of classes of the form
\begin{displaymath}
u_* q_i^* f^{P*} \pi_i^* \theta \qquad \text{and} \qquad u_* q_{ij}^* (f^P \times f^P)^* (\pi_i \times \pi_j)^* l_{J \times J}.
\end{displaymath}

\paragraph{Step $3$}

The class $f^{P*} \pi_i^* \theta$ is a divisor class (modulo algebraic equivalence) on the curve $C$. Thus it is a multiple of the class of a point. Therefore, $q_i^* f^{P*} \pi_i^* \theta$ is a multiple of the class $C \times \ldots \times C \times P \times C \times \ldots \times C$ (where the factor $P$ is in $i$th place). So we obtain that $u_* q_i^* f^{P*} \pi_i^* \theta$ is proportional to
\begin{displaymath}
(\pi_{1*} C) * \ldots * \widecheck{(\pi_{i*} C)} * \ldots * (\pi_{r*} C) \in \mathcal{F}(S_G)
\end{displaymath}
where the $\widecheck{~}$ means that we omit the emphasized factor.

\paragraph{Step $4$}

The main part of this proof rests in the study of classes $(f^P \times f^P)^* (\pi_i \times \pi_j)^* l_{J \times J}$. Put $\mathcal{M} := (f^P \times f^P)^* (\pi_i \times \pi_j)^* \mathcal{P}_{J \times J} \in \pic(C \times C)$. In order to study this invertible sheaf, we are going to study its fibres and then glue them.

Let $M$ be a point on $C$. Define $j_M: N \in C \mapsto (N,M) \in C^2$ and similarly $\widetilde{j}_V: U \in J \mapsto (U,V) \in J^2$ where $V$ is a given point on $J$. Then we easily check that
\begin{align*}
\mathcal{M}_{|C \times M} & \simeq j_M^* (f^P \times f^P)^* (\pi_i \times \pi_j)^* \mathcal{P}_{J \times J} \\
& \simeq f^{P*} \pi_i^* \widetilde{j}_{\pi_j f^P(M)}^*
\left( m^* \mathcal{L}_J(\Theta) \otimes p^* \mathcal{L}_J(\Theta)^\vee \otimes q^* \mathcal{L}_J(\Theta)^\vee \right) \\
& \simeq f^{P*} \pi_i^* \left( t_{\pi_j f^P(M)}^* \mathcal{L}_J(\Theta) \otimes \mathcal{L}_J(\Theta)^\vee \right).
\end{align*}
Therefore the isomorphism class of $\mathcal{M}_{|C \times M}$ corresponds to the point
\begin{displaymath}
f^{P*} \pi_i^* \varphi_{\Theta}(\pi_j f^P(M)) \in \pic^0(C).
\end{displaymath}
Then note that we have by definition of the Rosati involution $R$ the equality $\widehat{\pi_i} \circ \varphi_\Theta = \varphi_\Theta \circ R(\pi_i)$. Thus we get by Remark \ref{remConventionFp}, property 5
\begin{displaymath}
f^{P*} \varphi_{\Theta}(R(\pi_i) \pi_j f^P(M)) 
= f^{P*} \varphi_{\Theta}(\pi f^P(M)) = - \pi f^P(M)
\end{displaymath}
where $\pi := R(\pi_i) \circ \pi_j \in \mathbb{Z}[G]$. Since $\pi$ has integer coefficients and $\varphi_{\Theta}$ is a morphism of abelian varieties, it suffices essentially to study the case where $\pi$ is a monomial, that is $\pi = g \in G$. 

Moreover, we know from relations of Remark \ref{remConventionFp} that
\begin{displaymath}
(f^P \times f^P)^* \mathcal{P}_{J \times J} \simeq (\mathcal{L}^P)^\vee := \mathcal{L}_{C^2}( P \times C + C \times P - \Delta_C),
\end{displaymath}
and therefore we have by Albanese property
\begin{displaymath}
- g f^P(M) = - f^{g(P)}(g(M)) = \mathcal{L}^{g(P) \vee} _{|C \times g(M)}.
\end{displaymath}
It follows that for all points $M$ in $C$
\begin{align*}
\mathcal{M}_{|C \times M} 
& \simeq \mathcal{L}^{g(P) \vee} _{|C \times g(M)}
\simeq \left((1 \times g)^* \mathcal{L}^{g(P) \vee}\right)_{|C \times M}
\end{align*}
because $j_{g(M)} = (1 \times g) \circ j_M$.
According to the Seesaw principle \cite[Corollary 6 p54]{MR2514037}, we deduce the existence of a line bundle $\mathcal{N}$ on $C$ such that
\begin{align*}
\mathcal{M} & \simeq \left((1 \times g)^* \mathcal{L}^{g(P) \vee}\right) \otimes q^* \mathcal{N} 
\simeq \Big( (1 \times g)^* \mathcal{L}_{C^2} \big(g(P) \times C + C \times g(P) - \Delta_C \big) \Big) \otimes q^* \mathcal{N} 
\end{align*}
where $q: C \times C \to C$ denotes the second projection.
Therefore, passing to algebraic cycle classes, we get
that $(f^P \times f^P)^*(\pi_i \times \pi_j)^* l_{J \times J}$ is a linear combination of classes of the form
\begin{enumerate}
\item $(1 \times g)^* (g(P) \times C) = P \times C$ (because all points on the curve $C$ are algebraically equivalent and each $g \in G$ is still an automorphism),
\item $(1 \times g)^* (C \times g(P)) = C \times P$,
\item $(1 \times g)^* \Delta_C$,
\item $q^* (\deg(\mathcal{N}) \cdot P) = \deg(\mathcal{N}) ( C \times P )$
\end{enumerate}
for some automorphisms $g \in G$.

\begin{remark}
Note that there appear naturally cycle classes of the form $(1 \times g)^* \Delta_C$, that is essentially graphs of automorphisms.
\end{remark}

Finally, $q_{ij}^* (f^P \times f^P)^* (\pi_i \times \pi_j)^* l_{J \times J}$ is a linear combination of
\begin{enumerate}
\item $q_{ij}^*(P \times C) = q_i^* P$,
\item $q_{ij}^*(C \times P) = q_j^* P$,
\item $q_{ij}^* (1 \times g)^* \Delta_C$
\end{enumerate}
and so
$u_* q_{ij}^* (f^P \times f^P)^* (\pi_i \times \pi_j)^* l_{J \times J}$ is a linear combination of the following classes
\begin{enumerate}
\item $(\pi_{1*} C) * \ldots * \widecheck{(\pi_{i*} C)} * \ldots * (\pi_{r*} C) \in \mathcal{F}(S_G)$,
\item $(\pi_{1*} C) * \ldots * \widecheck{(\pi_{j*} C)} * \ldots * (\pi_{r*} C) \in \mathcal{F}(S_G)$,
\item $(\pi_{1*} C) * \ldots * \widecheck{(\pi_{i*} C)} * \ldots * \widecheck{(\pi_{j*} C)} * \ldots * (\pi_{r*} C) * (\underbrace{\pi_i + \pi_j g^{-1}}_{=: ~ \pi_{r+1}})_* C  \in \mathcal{F}(S_G)$.
\end{enumerate}

\paragraph{Conclusion}

So we proved that each cycle class $\theta \cdot \left[ (\pi_{1*} C) * \ldots * (\pi_{r*} C) \right] \in \mathcal{F}(S_G)$ defines an element in $\mathcal{F}(S_G)$ as (rational) linear combination of classes all belonging to the $\mathbb{Q}$-vector space $\mathcal{F}(S_G)$. Thus $\mathcal{F}(S_G)$ is stable under intersection with the (principal) polarization $\theta$, so stable under $\mathcal{F}$. Therefore this fact also holds for $S_G$, which we know now that is stable under Pontryagin product. This completes the proof of this key-theorem.
\end{proof}

\subsection{Interpretation in terms of tautological rings}

Theorem \ref{theoSsigmaiStablePontryagin} yields all we need to prove Theorem \ref{theoRsigmiG}. The hard part has already been done. It is now easy to conclude.

\begin{proof}[Proof of Theorem \ref{theoRsigmiG}]
By definition, tautological ring $R_G(C ; J)$ is the smallest $\mathbb{Q}$-vector subspace of $\dA(J)$ containing every $\pi_* C$ where $\pi \in \mathbb{Z}[G]$ and stable under intersection product, Pontryagin product and operators $k_*, k^*$. Therefore it contains the $\mathbb{Q}$-algebra (for the Pontryagin product) generated by thes classes $\pi_* C$. According to Theorem \ref{theoSsigmaiStablePontryagin} and Lemma \ref{lemGenFS''casRevetGeneral} this $\mathbb{Q}$-algebra is none other than $\mathcal{F}(S_G(C ; J))$, which equals $S_G(C ; J)$ thanks to Proposition \ref{propEquivStabFourierPontryaginGeneralisePhi}. So we have
\begin{displaymath}
R_G(C ; J) \supset S_G(C ; J).
\end{displaymath}
Also since $S_G(C ; J) = \mathcal{F}(S_G(C ; J))$ contains each $\pi_* C$ and is closed under intersection product, Pontryagin product and also under the operators $k_*$ and $k^*$ (because $S_G(C ; J)$ is generated by homogeneous classes $\pi^* N^i(w)$ in $\dA^i(J)_{(i-1)}$), one has
\begin{displaymath}
R_G(C ; J) \subset S_G(C ; J).
\end{displaymath}
So we get the equality $R_G(C ; J) = S_G(C ; J)$.
\end{proof}

We now have a tautological ring on $J$ associated to the group of automorphisms $G \subset \aut(C)$. This ring is all the more natural if one considers Corollary \ref{corPlusPetiteExtension}. Let us prove it now.

\begin{proof}[Proof of Corollary \ref{corPlusPetiteExtension}]
Here again we decompose the proof in several steps.

\paragraph{Step $1$}
The algebra $R(C ; J)$ introduced by Beauville is generated (for the intersection product) by $N^1(w) = \theta,\ldots,N^{g}(w)$ according to \cite{MR2041776} (and even by $N^1(w),\ldots,N^{g-1}(w)$). So if $S$ is an arbitrary algebra extension of $R(C ; J)$ stable under all pull-backs by polynomials in $\mathbb{Z}[G]$, then necessarily $\mathbb{Z}[G]^* R(C ; J) \subset S$ and in particular, $S$ contains each $\pi^* N^i(w)$ for all $\pi \in \mathbb{Z}[G]$ and $i \in \llbracket 1,g\rrbracket$. Thus, $R_G(C ; J) \subset S$.

\paragraph{Step $2$}
Since $R_G(C ; J)$ contains each $\pi^* N^i(w)$, it follows that $R_G(C ; J)$ contains each generator $N^i(w)$ of the algebra $R(C ; J)$. That is why we trivially have the inclusion $R(C ; J) \subset R_G(C ; J)$.

\paragraph{Step $3$}
Now $R_G(C ; J)$ is generated as $\mathbb{Q}$-vector space by classes of the form $\pi_1^* N^{i_1}(w) \cdot \ldots \cdot \pi_k^* N^{i_k}(w)$. Moreover, if we pull-back each one of these cycles by elements in $\mathbb{Z}[G]$, we still have a cycle in $R_G(C ; J)$ of the same form. Indeed, for all $\pi \in \mathbb{Z}[G]$ and all $x,y \in \dA(J)$, we have $\pi^* (x\cdot y) = \pi^* x \cdot \pi^* y$. Therefore, $R_G(C ; J)$ is stable under all pull-backs by elements in $\mathbb{Z}[G]$.

\paragraph{Conclusion}
In other words $R_G(C ; J)$ is an algebra extension of $R(C ; J)$ (for the intersection product) stable under all pull-backs by polynomials in $\mathbb{Z}[G] \subset \End(J)$ and contained in every extension $S$ of $R(C ; J)$ with this property. Thus $R_G(C ; J)$ is the smallest $\mathbb{Q}$-algebra extension of $R(C ; J)$ (for the intersection product) which is stable under pull-backs by polynomials in $\mathbb{Z}[G]$.

Similarly we prove that the tautological ring $R_G(C ; J)$ is the smallest $\mathbb{Q}$-algebra extension of $R(C ; J)$ (for the Pontryagin product) which is stable under push-forwards by polynomials in $\mathbb{Z}[G]$.
\end{proof}

This proves the following interpretation of the tautological ring $R_G(C ; J)$. It is the smallest $\mathbb{Q}$-vector subspace of $\dA(J)$ containing the cycle class $C$ and closed under intersection product, Pontryagin product, pull-backs and push-forwards by polynomials in $\mathbb{Z}[G]$. Actually, since the generators $\pi^* N^i(w)$ and $\pi_* C_{(i-1)}$ are homogeneous with respect to Beauville's decomposition, the tautological ring $R_G(C ; J)$ is even closed under pull-backs and push-forwards by polynomials in $\mathbb{Q}[G] \subset \End^0(J)$. 

\subsection{Tautological divisors, Néron-Severi group and symmetric endomorphisms}

Let $\sigma \in \aut(C)$ and $G = \langle \sigma \rangle \subset \aut(C)$. It is well-known that the Theta polarization induces an isomorphism  between the rational Néron-Severi group of $J$ and the set of symmetric endomorphisms of $J$ (see \cite[p190]{MR2514037}):
\begin{align*}
\ns_\mathbb{Q}(J) & \overset{\simeq}{\longrightarrow} \End^{(s)}(J) = \{f \in \End^0(J) ~|~ R(f) = f\} \\
D \quad & \mapsto \quad \varphi_{\Theta}^{-1} \circ \varphi_D.
\end{align*}
Under this bijection, for any $\pi \in \mathbb{Z}[G]$ the divisor class $\pi^* N^1(w) = \pi^* \theta \in R_{\sigma}(C ; J)$ corresponds to the symmetric endomorphism $R(\pi) \circ \pi$. Indeed, we easily check on points that $\varphi_{\pi^* \Theta} = \widehat{\pi} \circ \varphi_{\Theta} \circ \pi$. In particular, if $\pi$ is symmetric, then $\pi^* \theta$ corresponds to $\pi^2 \in \End^{(s)}(J)$.

For example, for any integer $i$ the divisor class $\gamma_i := (\sigma^i + \sigma^{-i})^* \theta$ is associated to the endomorphism $(\sigma^i + \sigma^{-i})^2 = \sigma^{2i} + \sigma^{-2i} + 2$. Also, let $\Gamma_i \in \dA^1(J)$ be the divisor class corresponding to $\sigma^{i} + \sigma^{-i} \in \End^{(s)}(J)$. We then have the relation
\begin{displaymath}
\gamma_i = \Gamma_{2i} + 2 \theta \in \dA^1(J) \cap R_{\sigma}(C ; J).
\end{displaymath}
We leave it to the reader to verify that these cycle classes $\gamma_i$ and $\Gamma_i$ are both related to the graphs $\Gamma_{\sigma^i}$ and $\Gamma_{\sigma^{-i}}$ in $\dA^1(C^2)$ of $\sigma^i$ and $\sigma^{-i}$. For example, we can obtain some relations of the form
\begin{displaymath}
f^{2*} \Gamma_i = k (C \times P) + k (P \times C) - \Gamma_{\sigma^i} - \Gamma_{\sigma^i} \in \dA^1(C^2)
\end{displaymath}
for some integer $k$ where $f^2 := m \circ (f^P \times f^P) : C \times C \to J$. This can be seen by using the proof of Theorem \ref{theoSsigmaiStablePontryagin}, Step 4.
More generally, any divisor class of the form $\pi^* \theta$ can be related to graphs of elements in $G$. Here again, we see how natural this tautological ring $R_\sigma(C ; J)$ is.

\section{The tautological ring $R_G(\psi_{Y*} C ; Y)$} \label{sec:ringY}

From now on and until the end of this paper we consider a $n$-cyclic Galois covering $f : C \to C' \simeq C/\langle\sigma\rangle$ associated to an automorphism $\sigma \in \aut(C)$ of finite order $n \in \mathbb{N}^*$. Moreover, we fix a finite automorphism group $G \subset \aut(C)$ and we suppose that each $g \in G$ commutes with $\sigma$. Therefore, each $g \in G$ determines an automorphism $\widetilde{g} \in \aut(C')$ fitting into the following commutative diagram
\begin{displaymath}
\xymatrix{
C \ar[r]^f \ar[d]_{g} & C' \ar[d]^{\widetilde{g}} \\
C \ar[r]_f & C'.
}
\end{displaymath}
Also, let us denote by $\widetilde{G}$ the subgroup of $\aut(C')$ formed by all these automorphisms $\widetilde{g}$ with $g \in G$.
The covering $f : C \to C'$ determines two complementary abelian subvarieties $Y$ and $Z$ of $J = J(C)$ as in \S \ref{sec:functoriality}. We recall that we denote by $\eta$ the polarization on $Y$ induced by $\Theta$ and
\begin{displaymath}
\psi_Y := \psi_{\eta} \circ \widehat{\iota_Y} \circ \varphi_\Theta \in \hom(J,Y)
\end{displaymath}
which is polynomial in $\sigma$ (according to Lemma \ref{lemExposantYGeneralise} (2)). In this section we study the tautological ring induced on $Y$ by $R_G(C ; J)$. This is the aim of Theorem \ref{theoRsigmaisYGeneralise}.

\begin{proof}[Proof of Theorem \ref{theoRsigmaisYGeneralise}]
In order to ease notations, put $J' = J(C')$. Since $j: J' \to Y$ is an isogeny, it induces an isomorphism (of $\mathbb{Q}$-vector spaces) between $R_{\widetilde{G}}(C' ; J') \subset \dA(J')$ and its image $j_* R_{\widetilde{G}}(C' ; J')$ in $\dA(Y)$. Also, there exists an isogeny $u: Y \to J'$ such that $u \circ j = k_{J'}$ and $j \circ u = k_Y$ for some integer $k \in \mathbb{N}^*$. According to Corollary \ref{corInversionIsogeniePullPushAlgebre}, we get
\begin{displaymath}
R_{\widetilde{G}}(C' ; J') 
 \simeq j_* R_{\widetilde{G}}(C' ; J') 
 = u^* R_{\widetilde{G}}(C' ; J').
\end{displaymath}
It follows that $j_* R_{\widetilde{G}}(C' ; J') 
 = u^* R_{\widetilde{G}}(C' ; J')$ is stable under intersection product, Pontryagin product but also under operators $k_*$ and $k^*$ (because $R_{\widetilde{G}}(C' ; J')$ has these properties, $j$ is a morphism of abelian varieties, pull-backs commute with the intersection product and push-forwards commute with the Pontryagin product). This proves that
\begin{align*}
R_{\widetilde{G}}(C' ; J') 
& \simeq j_* R_{\widetilde{G}}(C' ; J')
= \taut_Y\Big(j_* R_{\widetilde{G}}(C' ; J') \Big).
\end{align*}
According to Proposition \ref{propPushForwardNfFbarre}, we have $(N_f)_* C = n C'$. Thus, we get for all $\widetilde{\pi} \in \mathbb{Z}[\widetilde{G}]$
\begin{displaymath}
\widetilde{\pi}_* (N_f)_* C = n \widetilde{\pi}_* C'.
\end{displaymath}
But we have for all $g \in G$, $f \circ g = \widetilde{g} \circ f$ so that more generally when considering Albanese morphisms, we have for all $\pi \in \mathbb{Z}[G]$, $N_f \circ \pi = \widetilde{\pi} \circ N_f$ where if $\pi = \sum_{g \in G} a_g \circ g$, then $\widetilde{\pi} = \sum_{g \in G} a_g \circ \widetilde{g} \in \mathbb{Z}[\widetilde{G}]$ with $\widetilde{g}$ the automorphism induced on $C/\langle\sigma\rangle$ (or its Jacobian).
Consequently, we obtain
\begin{displaymath}
(N_f)_* \pi_* C = n \widetilde{\pi}_* C'.
\end{displaymath}
From this equality we deduce as in Corollary \ref{corPushForwardNfCi} that for each $i$, $(N_f)_* \pi_* C_{(i)} = n \widetilde{\pi}_* C_{(i)}'$. Then as in Corollary \ref{corSurjectiviteNf_*}, we get a surjection
\begin{displaymath}
(N_f)_*: R_{G}(C ; J) \relbar\joinrel\twoheadrightarrow R_{\widetilde{G}}(C' ; J') .
\end{displaymath}
Similarly, we obtain by Fourier transform a surjective morphism (see Corollary \ref{corSurjectivitePullBackFbarre})
\begin{displaymath}
\overline{f}^*: R_{G}(C ; J) \relbar\joinrel\twoheadrightarrow R_{\widetilde{G}}(C' ; J') .
\end{displaymath}
We now repeat the argument used in Proposition \ref{propImageR(C')dansY}: 
\begin{align*}
j_* R_{\widetilde{G}}(C' ; J') =  j_* \overline{f}^* R_{G}(C ; J) 
= j_* (\iota_Y \circ j)^* R_{G}(C ; J) = j_* j^* \iota_Y^* R_{G}(C ; J) = \iota_Y^* R_{G}(C ; J)
\end{align*}
because $j_* j^* = \deg(j) \cdot \id_{\dA(Y)}$. Accordingly,
\begin{align*}
R_{\widetilde{G}}(C' ; J') 
& \simeq j_* R_{\widetilde{G}}(C' ; J') = \iota_Y^* R_{G}(C ; J) = \taut_Y \Big(j_* R_{\widetilde{G}}(C' ; J') \Big).
\end{align*}
It remains to show the following equalities:
\begin{align*}
 \taut_Y \Big(j_* R_{\widetilde{G}}(C' ; J') \Big) = \psi_{Y*} R_G(C ; J)  = R_G(\psi_{Y*} C ; Y). 
\end{align*}
As the bigraded $\mathbb{Q}$-algebra $\iota_Y^* R_{G}(C ; J) = j_* R_{\widetilde{G}}(C' ; J')$ contains the induced polarization $\eta = \iota_Y^* \theta$ and is stable under both products, the assertions (2) and (3) of Proposition \ref{propEquivStabFourierPontryaginGeneralisePsi} prove that
\begin{displaymath}
\iota_Y^* R_{G}(C ; J) = \psi_{\eta*}\mathcal{F}_{Y} \psi_{\eta*}\mathcal{F}_Y \Big(\iota_Y^* R_{G}(C ; J) \Big) \underset{Prop. \ref{propEquivStabFourierPontryaginGeneralisePsi}}{=} \psi_{\eta*} \mathcal{F}_Y\Big(\iota_Y^* R_{G}(C ; J) \Big).
\end{displaymath}
Let us get a more explicit description of the generators $\psi_{\eta*} \mathcal{F}_Y\Big(\iota_Y^* R_{G}(C ; J) \Big)$. Using Proposition \ref{propCommutationFourierPullBackPushForward}, Lemma \ref{lemRosatiSigma} and the fact that $\psi_Y$ commute with any $\pi \in \mathbb{Z}[G]$, the following equalities are satisfied:
\begin{align*}
\psi_{\eta*} \mathcal{F}_Y(\iota_Y^* \pi^* N^i(w))
& = (-1)^{g- g'} \psi_{\eta*}\widehat{\iota_Y}_* \widehat{\pi}_* \mathcal{F}_J(N^i(w))
= (-1)^{g+i + (g- g')} \psi_{\eta*}\widehat{\iota_Y}_*  \widehat{\pi}_* \varphi_{\Theta*} C_{(i-1)} \\
& = (-1)^{i - g'} \psi_{\eta*}\widehat{\iota_Y}_*  \varphi_{\Theta*} \varphi_{\Theta*}^{-1} \widehat{\pi}_* \varphi_{\Theta*} C_{(i-1)}
= (-1)^{i - g'} \psi_{\eta*}\widehat{\iota_Y}_*  \varphi_{\Theta*} R(\pi)_* C_{(i-1)} \\
& = (-1)^{i - g'} (\psi_{\eta}\widehat{\iota_Y}\varphi_{\Theta})_* R(\pi)_* C_{(i-1)}
= (-1)^{i- g'} \psi_{Y*} R(\pi)_*C_{(i-1)} \\
& = (-1)^{i-g'}  R(\pi)_* \psi_{Y*} C_{(i-1)}
\end{align*}
where we recall that the Rosati involution $R$ induces a surjection $R : \mathbb{Z}[G] \to \mathbb{Z}[G]$ (see the proof of Theorem \ref{theoSsigmaiStablePontryagin}, Step 1).
It follows that $\psi_{\eta*} \mathcal{F}_Y\Big(\iota_Y^* R_{G}(C ; J) \Big)$ is generated as $\mathbb{Q}$-vector space by the products of the form
\begin{displaymath}
(\pi_{1*} \psi_{Y*} C_{(i_1 -1)})* \ldots * (\pi_{r*} \psi_{Y*} C_{(i_r -1)}).
\end{displaymath}
Actually, using the argument of the proof of Lemma \ref{lemGenFS''casRevetGeneral}, we get that $\psi_{\eta*} \mathcal{F}_Y \Big(\iota_Y^* R_{G}(C ; J) \Big)$ is the algebra (for the Pontryagin product) generated by all $\pi_* \psi_{Y*} C$ for polynomials $\pi \in \mathbb{Z}[G]$.
Since the Pontryagin product commutes with push-forwards and as we already noted that $\psi_{Y*}$ (which is a polynomial in $\sigma$) commutes with each $\pi \in \mathbb{Z}[G]$, we immediately get the equalities
\begin{displaymath}
\taut_Y \Big(j_* R_{\widetilde{G}}(C' ; J') \Big) = \iota_Y^* R_{G}(C ; J) = \psi_{\eta*} \mathcal{F}_Y \Big(\iota_Y^* R_{G}(C ; J) \Big) = \psi_{Y*} R_{G}(C ; J).
\end{displaymath}
Moreover this shows not only that
\begin{align*}
\taut_Y \Big(j_* R_{\widetilde{G}}(C' ; J') \Big)
 \supset \taut_Y \Big(\{ \pi_* \psi_{Y*} C\in \dA(Y) ~|~ \pi \in \mathbb{Z}[G]\} \Big) =: R_G(\psi_{Y*}C ; Y)
\end{align*}
but also the reverse inclusion. Indeed,
\begin{displaymath}
R_G(\psi_{Y*}C ; Y) := \taut_Y \Big(\{ \pi_* \psi_{Y*} C\in \dA(Y) ~|~  \pi \in \mathbb{Z}[G]\} \Big)
\end{displaymath}
contains (by definition) the algebra for the Pontryagin product generated by all $\pi_* \psi_{Y*} C$, which equals $\taut_Y \Big(j_* R_{\widetilde{G}}(C' ; J')\Big)$. This completes the proof of Theorem \ref{theoRsigmaisYGeneralise}.
\end{proof}

In fact, it is quite reasonable that we were able to deduce  Theorem \ref{theoRsigmaisYGeneralise} from Theorem \ref{theoRsigmiG} since $Y$ is closely related to the Jacobian $J(C') \simeq J(C/\langle\sigma\rangle)$ on which all technical issues have been solved previously.
Also note that if we denote by $H \subset \aut(C)$ the group generated by $\sigma$ and $G$ (in such a way that $\sigma$ is central in $H$), then we have
\begin{displaymath}
\iota_Y^* R_{H}(C ; J) = \iota_Y^* R_{G}(C ; J).
\end{displaymath}
Indeed, we have $\mathbb{Z}[H] = \mathbb{Z}[G]$ as subrings of $\End(Y)$ since $Y = \ker(\sigma -1)^0$ (according to Lemma \ref{lemExposantYGeneralise}). 
This remark is not true in general in the next section in which we will work with cycles supported on $Z = \ker(N_Y)^0$. Thus we will need to consider (in general) elements in $\mathbb{Z}[H]$ and not only in $\mathbb{Z}[G]$.

\section{The tautological ring $R^{\sigma}_H(\psi_{Z*} C ; Z)$} \label{sec:ringZ}

We want to obtain a tautological ring in $\dA(Z)$ as we just did in $\dA(Y)$. The basic strategy remains identical to Theorem \ref{theoSsigmaiStablePontryagin} except that we have to manage the fact that $Z$ has no reason to be (isogenous to) a Jacobian. Nevertheless, considering the induced polarization on $Z$, also denoted by $\eta := \iota_Z^* \theta \in \dA^1(Z)$, and noting that this polarization is closely related to canonical principal polarizations of $J = J(C)$ and $J' = J(C') \simeq J(C/\langle\sigma\rangle)$ (associated to $\theta \in \dA^1(J)$ and $\theta' \in \dA^1(J')$), we can solve this problem. As in the previous section, put $\psi_Z := \psi_\eta \circ \widehat{\iota_Z} \circ \varphi_\Theta$. We have $N_Z = \iota_Z \circ \psi_Z$ and these morphisms are polynomials in $\sigma$. Also consider a finite automorphism group $H \subset \aut(C)$ such that $\sigma \in H$ is central.

\subsection{Key-theorem}

\begin{theorem} \label{theoSsigmaiStablePontryaginZGeneralise}
Let $S^{\sigma}_{H} = S^{\sigma}_{H}(\psi_{Z*} C ; Z)$ be the $\mathbb{Q}$-subalgebra of $\dA(Z)$ (for the intersection product) generated by the $\pi^* \iota_Z^* N^i(w)$ for $\pi \in \mathbb{Z}[H]$ and $i \in \llbracket 1,\dim Z-1\rrbracket$. Then  $S^{\sigma}_{H}$ is stable under the Pontryagin product.
\end{theorem}

\begin{remark}
Note here the particular role played by $\sigma$ because $Z$ depends on $\sigma$ and, in general, $\sigma$ is non-trivial in $\aut(Z)$.
\end{remark}

\begin{proof}
As for Theorem \ref{theoSsigmaiStablePontryagin}, we decompose the proof in several steps.

\paragraph{Step $1$}

According to Theorem \ref{theoRsigmiG}, the algebra $S^{\sigma}_{H}$ is none other than the restriction $\iota_Z^* R_{H}(C ; J)$ of the corresponding tautological ring on $J$.
The strategy of the proof is again to use implication $(4) \Rightarrow (1)$ of Proposition \ref{propEquivStabFourierPontryaginGeneralisePsi} (note that Proposition \ref{propEquivStabFourierPontryaginGeneralisePsi} applies because $\eta := \iota_Z^* \theta = \iota_Z^* N^1(w) \in S^{\sigma}_{H}$). Thus we have to show that
\begin{displaymath}
\eta \cdot \psi_{\eta*} \mathcal{F}_Z(S^{\sigma}_{H}) \subset \psi_{\eta*} \mathcal{F}_Z(S^{\sigma}_{H}).
\end{displaymath}
Using the same arguments as in the proof of Theorem \ref{theoRsigmaisYGeneralise}, we get
\begin{align*}
& \psi_{\eta*} \mathcal{F}_Z(\iota_Z^* \pi^* N^i(w)) 
= (-1)^{i-\dim Z} R(\pi)_* \psi_{Z*} C_{(i-1)}.
\end{align*}
It follows that $\psi_{\eta*}\mathcal{F}_Z(S^{\sigma}_{H})$ is generated as $\mathbb{Q}$-vector space by products of the form
\begin{displaymath}
(\pi_{1*} \psi_{Z*} C_{(i_1 -1)})* \ldots * (\pi_{r*} \psi_{Z*} C_{(i_r -1)})
\end{displaymath}
hence by products of the form $(\pi_{1*} \psi_{Z*} C)* \ldots * (\pi_{r*} \psi_{Z*} C)$ with nonzero $\pi_j \in \mathbb{Z}[H]$.

\paragraph{Step $2$}

For any nonzero element $\pi_j \in \mathbb{Z}[H]$ we study the class of the cycle
\begin{displaymath}
\eta \cdot \left[(\pi_{1*} \psi_{Z*} C) * \ldots * (\pi_{r*} \psi_{Z*} C)\right].
\end{displaymath}
The same argument as in the proof of Theorem \ref{theoSsigmaiStablePontryagin} shows that this cycle is a linear combination of elements of the form
\begin{displaymath}
u_* q_i^* f^{P*} \psi_Z^* \pi_i^* \eta \qquad \text{and} \qquad u_* q_{ij}^* (f^P \times f^P)^* (\psi_Z \times \psi_Z)^* (\pi_i \times \pi_j)^* l_{Z \times Z}
\end{displaymath}
where 
\begin{enumerate}
\item $l_{Z \times Z} := (1 \times \varphi_\eta)^* l_{Z \times \widehat{Z}} = m^* \eta - p^* \eta - q^* \eta \in \dA^1(Z \times Z)$
\item the map $u: C^r \to Z$ is defined by the composition
\begin{displaymath}
u: C^r \overset{\Phi}{\longrightarrow} J^r \overset{\Psi}{\longrightarrow} Z^r \overset{K}{\longrightarrow} Z^r \overset{m}{\longrightarrow} Z
\end{displaymath}
with $\Phi := f^P \times \ldots \times f^P$ ($r$ times), $\Psi := \psi_Z \times \ldots \times \psi_Z$, $K := \pi_1 \times \ldots \times \pi_r$ and where $m : Z^{r} \to Z$ is induced by the multiplication on $Z$.
\end{enumerate}

\paragraph{Step $3$}

Using the same argument as in the proof of Theorem \ref{theoSsigmaiStablePontryagin}, Step $3$, we immediately conclude that $u_* q_i^* f^{P*} \psi_Z^* \pi_i^* \eta \in \psi_{\eta*}\mathcal{F}_Z(S^{\sigma}_{H})$.

\paragraph{Step $4$}

We now have to study the class $(f^P \times f^P)^* (\psi_Z \times \psi_Z)^* (\pi_i \times \pi_j)^* l_{Z \times Z}$. Since $\sigma$ is central in $H$, we have
\begin{displaymath}
(\pi_i \times \pi_j) \circ (\psi_Z \times \psi_Z) = (\psi_Z \times \psi_Z) \circ (\pi_i \times \pi_j).
\end{displaymath}
Thus we obtain
\begin{displaymath}
(f^P \times f^P)^* (\psi_Z \times \psi_Z)^* (\pi_i \times \pi_j)^* l_{Z \times Z} 
= (f^P \times f^P)^* (\pi_i \times \pi_j)^* (\psi_Z \times \psi_Z)^* l_{Z \times Z}.
\end{displaymath}
The key-argument is the following (see \cite[Proposition 12.3.4]{MR2062673}):
\begin{align*}
e(Y)^2 \theta & = e(Y)^* \theta = N_Y^* \theta + N_Z^* \theta
= \left(\frac{e(Y)}{n} \overline{f} N_f\right)^* \theta + (\iota_Z \psi_Z)^* \theta \\
& = \frac{e(Y)^2}{n^2} N_f^* \overline{f}^* \theta + \psi_Z^* \eta
= \frac{e(Y)^2}{n} N_f^* \theta'+ \psi_Z^* \eta.
\end{align*}
These different equalities are justified by facts recalled in Section \ref{subsec:notations}; namely:
\begin{displaymath}
N_Y + N_Z = e(Y) \cdot \id_J, \qquad
\overline{f} N_f = \frac{n}{e(Y)} N_Y \qquad \text{and} \qquad
\overline{f}^* \theta = n \theta' \in \dA^1(J').
\end{displaymath}
In other words, we have
\begin{displaymath}
\psi_Z^* \eta = e(Y)^2 \theta - \frac{e(Y)^2}{n} N_f^* \theta'.
\end{displaymath}
Thus,
\begin{align*}
& (f^P \times f^P)^* (\pi_i \times \pi_j)^* (\psi_Z \times \psi_Z)^* l_{Z \times Z} \\
= \quad & e(Y)^2 (f^P \times f^P)^* (\pi_i \times \pi_j)^* l_{J \times J} - \frac{e(Y)^2}{n} (f^P \times f^P)^* (\pi_i \times \pi_j)^* (N_f \times N_f)^* l_{J' \times J'}.
\end{align*}

\paragraph{Step $5$}

The cycle class $(f^P \times f^P)^* (\pi_i \times \pi_j)^* l_{J \times J}$ has already been studied in the proof of Theorem \ref{theoSsigmaiStablePontryagin}. It is a linear combination of $P \times C$, $C \times P$ and $(1 \times h)^* \Delta_C$ for $h \in H$.
Using similar arguments as in the last part of the proof of Theorem \ref{theoSsigmaiStablePontryagin}, we show that $u_* q_{ij}^* (f^P \times f^P)^* (\pi_i \times \pi_j)^* l_{J \times J}$ is a linear combination of cycles in $\psi_{\eta*} \mathcal{F}_Z(S^{\sigma}_{H})$.

\paragraph{Step $6$}

It remains to study the class $(f^P \times f^P)^* (\pi_i \times \pi_j)^* (N_f \times N_f)^* l_{J' \times J'}$. As in \cite{MR2923946}, we use the equality $N_f \circ f^P = f^{f(P)} \circ f$. Nevertheless, we first have to commute the map $N_f$ with polynomials in $\mathbb{Z}[H]$. To be more precise, since by hypothesis $\sigma$ is central in $H$, there exist for all $h \in H$ an automorphism $\widetilde{h} \in \aut(C')$ such that $f \circ h = \widetilde{h} \circ f$.
These automorphisms $\widetilde{h}$ extend to automorphisms of $J(C')$ which we still denote by $\widetilde{h}$. We consider the group $\widetilde{H}$ formed by these automorphisms.

\begin{remark}
Note that we have $\widetilde{\sigma}  = 1_{J(C')}$.
\end{remark}

As in Section \ref{sec:ringY}, each $\pi \in \mathbb{Z}[H]$ induces an element $\widetilde{\pi} \in \mathbb{Z}[\widetilde{H}]$ and for all $k$, we have the relation $N_f \circ \pi_k = \widetilde{\pi_k} \circ N_f$. Thus we have
\begin{align*}
(f^P \times f^P)^* (\pi_i \times \pi_j)^* (N_f \times N_f)^* l_{J' \times J'}
& = (f^P \times f^P)^* (N_f \times N_f)^* (\widetilde{\pi_i} \times \widetilde{\pi_j})^*  l_{J' \times J'} \\
& = (f \times f)^* (f^{f(P)} \times f^{f(P)})^* (\widetilde{\pi_i} \times \widetilde{\pi_j})^*  l_{J' \times J'}. 
\end{align*}
Now the same argument as in Step 5 shows that this cycle class is a linear combination of
\begin{enumerate}
\item $(f \times f)^* (f(P) \times C') = n(P \times C)$ (because all points on $C$ are algebraically equivalent and because $f: C \to C'$ is of degree $n$),
\item $(f \times f)^* (C' \times f(P)) = n (C \times P)$ for the same reason,
\item and finally, since each $\widetilde{h} \in \widetilde{H} \subset \aut(C')$ is induced by some $h \in H \subset \aut(C)$, we also have cycles of the form
\begin{align*}
 (f \times f)^* (1 \times \widetilde{h})^* \Delta_{C'} 
 & = (1 \times h)^* (f \times f)^* \Delta_{C'} \\
 & = (1 \times h)^* (\Delta_C + (1 \times \sigma)^* \Delta_{C} + \ldots +  (1 \times \sigma^{n -1})^* \Delta_{C})
\end{align*}
for some elements $h \in H$.
\end{enumerate}
So far, we proved that $u_* q_{ij}^* (f^P \times f^P)^* (\pi_i \times \pi_j)^* (N_f \times N_f)^* l_{J' \times J'}$ is a linear combination of cycles all in $\psi_{\eta*}\mathcal{F}_Z(S^{\sigma}_{H})$. To be precise, the classes we obtain are on the one hand of the form
\begin{displaymath}
(\pi_{1*} \psi_{Z*} C) * \ldots * \widecheck{(\pi_{i*} \psi_{Z*} C)} * \ldots * (\pi_{r*} \psi_{Z*} C)
\end{displaymath}
and on the other, we get for $k \in \llbracket 0,n -1\rrbracket$ classes of the form:
\begin{align*}
(\pi_{1*} \psi_{Z*} C) * \ldots & * \widecheck{(\pi_{i*} \psi_{Z*} C)} * \ldots * \widecheck{(\pi_{j*} \psi_{Z*} C)} * \ldots * (\pi_{r*} \psi_{Z*} C) * (\pi_i + \pi_j h^{-1} \sigma^{-k})_* \psi_{Z*} C.
\end{align*}

\paragraph{Conclusion}

All this implies that $\eta \cdot \left[ (\pi_{1*} \psi_{Z*} C)* \ldots * (\pi_{r*} \psi_{Z*} C) \right] \in \psi_{\eta*}\mathcal{F}_Z(S^{\sigma}_{H})$ as (rational) linear combination of cycles which all belong to $\psi_{\eta*}\mathcal{F}_Z(S^{\sigma}_{H})$. In other words, $\psi_{\eta*}\mathcal{F}_Z(S^{\sigma}_{H})$ is stable by intersection with $\eta$ so that Proposition \ref{propEquivStabFourierPontryaginGeneralisePsi} shows that $S^{\sigma}_{H}$ is stable under Pontryagin product. This completes the proof.
\end{proof}

\subsection{Interpretation in terms of tautological rings}

Theorem \ref{theoSsigmaiStablePontryaginZGeneralise} immediately implies Theorem \ref{theoRsigmaisZGeneralise}, just as we deduced Theorem \ref{theoRsigmiG} from Theorem \ref{theoSsigmaiStablePontryagin}. At the same time, we obtain Theorem \ref{theoRsigmaisZGeneraliseSigma0}. \qed

\vspace*{5pt}

We now consider some special cases of Theorem \ref{theoRsigmaisZGeneralise}. We first consider the case where $\sigma$ is of order $2$. In that case
\begin{displaymath}
Z = \ker(N_Y)^0 = \ker(\Phi_2(\sigma))^0 = \ker(1 + \sigma)^0
\end{displaymath}
and thus $\sigma_{|Z} = -1_Z$. This implies that the image of $\mathbb{Z}[H]$ in $\End(Z)$ does not depend on $\sigma$. Therefore, as in Section \ref{sec:ringY}, $\sigma$ does not need to belong to $H$ : we just have to assume that each automorphism of $H$ commutes with $\sigma$. That being said, Theorem \ref{theoRsigmaisZGeneralise} leads to

\begin{theorem} \label{theoAnneauTautZPrym}
Let $f: C \to C' \simeq C/\langle\sigma\rangle$ be a double covering. In particular, this implies that $Z = \ker(1 + \sigma)^0$ and $\sigma_{|Z} = -1_Z$. We consider a finite group of automorphisms $G \subset \aut(C)$ and we suppose that each $g \in G$ commutes with $\sigma$. Then the tautological ring $R^{\sigma}_{G}(\psi_{Z*} C ; Z)$ is generated as $\mathbb{Q}$-subalgebra of $\dA(Z)$
\begin{enumerate}
\item for the intersection product by all $\pi^* \iota_Z^* N^i(w) = \iota_Z^* \pi^* N^i(w)$,
\item for the Pontryagin product by all $\pi_* \psi_{Z*} C_{(i-1)} = \psi_{Z*} \pi_* C_{(i-1)}$
\end{enumerate}
with $\pi \in \mathbb{Z}[G]$ and $i \in \llbracket 1,\dim Z-1 \rrbracket$ odd. As a result, we get the tautological ring:
\begin{displaymath}
R^{\sigma}_{G}(\psi_{Z*} C ; Z) = \iota_Z^* R_{G}(C ; J) = \psi_{Z*} R_{G}(C ; J).
\end{displaymath}
\end{theorem}

Note that in this theorem we can restrict to consider odd indices $i$ because of

\begin{lemma} \label{lemCourbeAbelPrymSymetrique}
With the above notations and assumptions of Theorem \ref{theoAnneauTautZPrym}, the cycle class of $\psi_{Z*} C \in \dA(Z)$ is symmetric. Therefore, each $\psi_{Z*} C_{(2i+1)} = 0$ in $\dA^{\dim Z -1}(Z)_{(2i+1)}$.
\end{lemma}

\begin{proof}
The following diagram is commutative
\begin{displaymath}
\xymatrix{
C \ar[r]^{f^P} \ar[d]_{\sigma} & J \ar[r]^{\psi_Z} \ar[d]^{\sigma} & Z \ar[d]^{-1} \\
C \ar[r]_{f^{\sigma(P)}} & J \ar[r]_{\psi_Z} & Z.
}
\end{displaymath}
Indeed commutativity of the left square follows from the definition of Albanese morphism. Whereas commutativity of the right hand square is justified by $\sigma_{|Z} = -1_{Z}$ (because $Z = \ker(1+\sigma)^0$). A diagram chase gives $(-1)_* \psi_{Z*} f^P_* C = \psi_{Z*} f^{\sigma(P)}_* \sigma_* C$. Since $(-1)^* = (-1)_*: \dA(Z) \to \dA(Z)$ and $\sigma_* C = C$ (because $\sigma \in \aut(C)$), we have $(-1)^* \psi_{Z*} f^P_* C = \psi_{Z*} f^{\sigma(P)}_* C$. But we are working modulo algebraic equivalence so that we can translate cycles without changing the cycle class $f^{\sigma(P)}_* C = f^{P}_* C$ in $\dA^{g-1}(J(C))$. Thus we have $(-1)^* \psi_{Z*} f^P_* C = \psi_{Z*} f^{P}_* C$, which means that $\psi_{Z*} f^P_* C$ (which we denoted by $\psi_{Z*} C$) is symmetric. Therefore (see \cite[Corollary 1]{MR726428})
\begin{displaymath}
\psi_{Z*} f^P_* C \in \bigoplus_i \dA^{\dim Z - 1}(Z)_{(2i)}
\end{displaymath}
and we have proven our lemma.
\end{proof}

Also, if one only considers the automorphism $\sigma$ of order $2$ (that is if we are interested in the restriction of Beauville's tautological ring $R(C ; J)$ to the subvariety $Z = \ker(1+ \sigma)^0$), we immediately deduce:

\begin{corollary}
Let $f: C \to C' \simeq C/\langle\sigma\rangle$ be a double covering. Then the tautological ring
\begin{displaymath}
R_{\sigma}(\psi_{Z*} C ; Z) := \taut_Z\Big(\{P(\sigma)_* \psi_{Z*} C \in \dA(Z) ~|~ P \in \mathbb{Z}[X]\}\Big)
\end{displaymath}
is generated as $\mathbb{Q}$-subalgebra of $\dA(Z)$
\begin{enumerate}
\item for the intersection product by all $\iota_Z^* N^i(w)$,
\item for the Pontryagin product by all $\psi_{Z*} C_{(i-1)}$
\end{enumerate}
for all odd indices $i \in \llbracket 1,\dim Z-1 \rrbracket$. Consequently, we get the tautological ring
\begin{displaymath}
R_{\sigma}(\psi_{Z*} C ; Z) = \iota_Z^* R(C ; J) = \psi_{Z*} R(C ; J).
\end{displaymath}
\end{corollary}

This corollary provides a generalization of Arap's theorem \cite[Theorem 4]{MR2923946}. His result deals with double coverings which are étale or ramified in exactly two points so that $Z$ is in fact a Prym variety (and in particular principally polarized which simplifies the proofs of Propositions \ref{propEquivStabFourierPontryaginGeneralisePhi} and \ref{propEquivStabFourierPontryaginGeneralisePsi}).

\subsection{Some remarks about relations between generators in $\dA(Y)$ and $\dA(Z)$}

Until now we studied tautological rings on $J$, $Y$ and $Z$. These rings on $Y$ and $Z$ are obtained as restrictions of analogous tautological rings on $J$. In particular, we can deduce relations between generators in $\dA(Y)$ or $\dA(Z)$ by projecting known relations in $\dA(J)$. By projecting we essentially mean applying $\iota_Y^*$ or $\iota_Z^*$ (resp. $\psi_{Y*}$ or $\psi_{Z*}$) if one considers relations for the intersection product (resp. Pontryagin product). 

\vspace*{5pt}

We recall a theorem of Colombo and van Geemen \cite{MR1241954} which states that if $C$ is a $k$-gonal curve, then $C_{(i)} = 0$ for all $i \geq k-1$. This is equivalent to $N^i(w) = 0$ for all $i \geq k$. Therefore in all previous results involving classes $N^i(w)$ we could restrict ourself to indices $i \in \llbracket 1, \text{gon}(C)-1 \rrbracket$ where $\text{gon}(C)$ is the smallest positive integer $d$ such that there exists a finite morphism of degree $d$ from $C$ to $\mathbb{P}^1$.

\vspace*{5pt}

Thus we can obtain two corollaries as in \cite{MR2041776} for tautological ring $R(C ; J)$ for hyperelliptic and trigonal curves. These corollaries explain the $\mathbb{Q}$-algebra structure (for the intersection product) of tautological rings $R_\sigma(\psi_{Z*} C ; Z) \subset \dA(Z)$ whence $\sigma$ is of order $2$ for $k$-gonal curves with $k \in \{2,3,4,5\}$.

\begin{corollary} \label{corRel1}
Let $f: C \to C' \simeq C/\langle\sigma\rangle$ be a double covering. We suppose that $C$ is hyperelliptic or trigonal and we denote by $\eta := \iota_Z^* \theta$ the induced polarization on $Z$. Also put $d := \dim Z$. Then
\begin{displaymath}
R_\sigma(\psi_{Z*} C ; Z) = \mathbb{Q}[\eta]/(\eta^{d+1}).
\end{displaymath}
\end{corollary}

\begin{proof}
If $C$ is hyperelliptic, then the only nonzero $N^i(w)$ is $N^1(w) = \theta$. Thus $R_\sigma(\psi_{Z*} C ; Z)$ is generated by $\iota_Z^* \theta = \eta$.
If $C$ is trigonal, the only nonzero generators are $N^1(w) = \theta$ and $N^2(w)$. However, projections $\iota_Z^* N^{2i}(w)$ of $N^{2i}(w)$ in $\dA(Z)$ are $0$ in $\dA(Z)$ (because the projection $\psi_{Z*} C = \psi_{Z*} C_{(0)} + \psi_{Z*} C_{(1)}$ is symmetric according to Lemma \ref{lemCourbeAbelPrymSymetrique}).
\end{proof}

The next corollary can be proven by closely following the strategy used by Beauville in \cite[Corollary 5.3]{MR2041776}.

\begin{corollary} \label{corRel2}
Let $f: C \to C' \simeq C/\langle\sigma\rangle$ be a double covering. We suppose that $C$ is $4$-gonal or $5$-gonal. We put $\eta := \iota_Z^* \theta \in \dA^1(Z)_{(0)}$ and $\mu := \iota_Z^* N^3(w) \in \dA^3(Z)_{(2)}$. We continue to write $d := \dim Z$. Then $R_\sigma(\psi_{Z*} C ; Z) \subset \dA(Z)$ is the algebra generated by $\eta$ and $\mu$ (for the intersection product). Moreover, there exists a positive integer $k \leq \frac{d}{5}$ such that
\begin{displaymath}
R_\sigma(\psi_{Z*} C ; Z) = \mathbb{Q}[\eta,\mu]/(\eta^{d+1},\eta^{d-4}\mu,\ldots,\eta^{d-5k+1}\mu^k,\mu^{k+1}).
\end{displaymath}
\end{corollary}

\begin{remark}
Since on any curve of genus $g$ there exists a $g^1_d$ with $d \leq \lfloor \frac{g+3}{2} \rfloor$ (see \cite[Chapter 5, Theorem 1.1]{MR770932}), Corollary \ref{corRel1} applies (in particular) to curves of genus $\leq 4$. Similarly, Corollary \ref{corRel2} applies to curves of genus $g \leq 8$.
\end{remark}

\subsection{Outlooks}

In general, finding a (complete) system of non-trivial relations between the $\pi^* N^i(w)$ is a hard task. Actually, it is already tough to study relations between the $N^i(w)$ (or the $C_{(i)}$) as shown by papers by Polishchuk, Colombo and van Geemen, and Herbaut. Also it would be interesting to lift these tautological rings modulo rational equivalence as it has been done for $R(C ; J)$ by Polishchuk. Furthermore, there is another important matter which would deserve to be studied. We know that different automorphism groups may determine the same tautological ring (e.g. on a hyperelliptic curve $C$ endowed with its hyperelliptic involution $\iota$, consider the trivial group $\{\id\}$ and $G = \{\id,\iota\}$). However, we can wonder whether non-isomorphic group algebras $\mathbb{Z}[G_1]$ and $\mathbb{Z}[G_2]$ (seen as subrings of $\End(J)$) always determine non-isomorphic tautological rings $R_{G_1}(C ; J)$ and $R_{G_2}(C ; J)$.

\vspace*{10pt}

\noindent \textbf{Acknowledgements} This article is part of my thesis written at the University of Strasbourg. I would like to thank my Ph.D. advisor Professor Rutger Noot for his continual support and numerous suggestions which contributed to the improvement of this paper. I also want to thank the referee for his comments which helped me to reduce the length of this paper.

\bibliographystyle{acm}

\begin{thebibliography}{10}

\bibitem{MR2923946}
{\sc Arap, M.}
\newblock Algebraic cycles on {P}rym varieties.
\newblock {\em Math. Ann. 353}, 3 (2012), 707--726.

\bibitem{MR770932}
{\sc Arbarello, E., Cornalba, M., Griffiths, P.~A., and Harris, J.}
\newblock {\em Geometry of algebraic curves. {V}ol. {I}}, vol.~267 of {\em
  Grundlehren der Mathematischen Wissenschaften [Fundamental Principles of
  Mathematical Sciences]}.
\newblock Springer-Verlag, New York, 1985.

\bibitem{MR726428}
{\sc Beauville, A.}
\newblock Quelques remarques sur la transformation de {F}ourier dans l'anneau
  de {C}how d'une vari\'et\'e ab\'elienne.
\newblock In {\em Algebraic geometry ({T}okyo/{K}yoto, 1982)}, vol.~1016 of
  {\em Lecture Notes in Math.} Springer, Berlin, 1983, pp.~238--260.

\bibitem{MR826463}
{\sc Beauville, A.}
\newblock Sur l'anneau de {C}how d'une vari\'et\'e ab\'elienne.
\newblock {\em Math. Ann. 273}, 4 (1986), 647--651.

\bibitem{MR2041776}
{\sc Beauville, A.}
\newblock Algebraic cycles on {J}acobian varieties.
\newblock {\em Compos. Math. 140}, 3 (2004), 683--688.

\bibitem{MR2062673}
{\sc Birkenhake, C., and Lange, H.}
\newblock {\em Complex abelian varieties}, second~ed., vol.~302 of {\em
  Grundlehren der Mathematischen Wissenschaften [Fundamental Principles of
  Mathematical Sciences]}.
\newblock Springer-Verlag, Berlin, 2004.

\bibitem{MR1241954}
{\sc Colombo, E., and van Geemen, B.}
\newblock Note on curves in a {J}acobian.
\newblock {\em Compositio Math. 88}, 3 (1993), 333--353.

\bibitem{MR999313}
{\sc Drezet, J.-M., and Narasimhan, M.~S.}
\newblock Groupe de {P}icard des vari\'et\'es de modules de fibr\'es
  semi-stables sur les courbes alg\'ebriques.
\newblock {\em Invent. Math. 97}, 1 (1989), 53--94.

\bibitem{MR1644323}
{\sc Fulton, W.}
\newblock {\em Intersection theory}, second~ed., vol.~2 of {\em Ergebnisse der
  Mathematik und ihrer Grenzgebiete. 3. Folge. A Series of Modern Surveys in
  Mathematics [Results in Mathematics and Related Areas. 3rd Series. A Series
  of Modern Surveys in Mathematics]}.
\newblock Springer-Verlag, Berlin, 1998.

\bibitem{MR861976}
{\sc Milne, J.~S.}
\newblock Jacobian varieties.
\newblock In {\em Arithmetic geometry ({S}torrs, {C}onn., 1984)}. Springer, New
  York, 1986, pp.~167--212.

\bibitem{MR2514037}
{\sc Mumford, D.}
\newblock {\em Abelian varieties}, vol.~5 of {\em Tata Institute of Fundamental
  Research Studies in Mathematics}.
\newblock Published for the Tata Institute of Fundamental Research, Bombay; by
  Hindustan Book Agency, New Delhi, 2008.
\newblock With appendices by C. P. Ramanujam and Yuri Manin, Corrected reprint
  of the second (1974) edition.

\end{thebibliography}

\end{document}